\documentclass{article}
\usepackage{amsmath,amsthm,amssymb,color,enumerate}
\usepackage{array}
\usepackage{graphicx}
\usepackage{caption}
\usepackage{subcaption}
\usepackage{hyperref}
\usepackage{float}
\usepackage[all]{xy}

\hypersetup{colorlinks=true,linkcolor=blue}
\newcommand{\fh}{\mathfrak{h}}

\newcommand{\bC}{\mathbb{C}}

\newcommand{\bN}{\mathbb{N}}
\newcommand{\cO}{\mathcal{O}}

\newcommand{\bP}{\mathbb{P}}
\newcommand{\bR}{\mathbb{R}}
\newcommand{\ft}{\mathfrak{r}}
\newcommand{\cY}{\mathcal{Y}}

\newcommand{\bZ}{\mathbb{Z}}

\newcommand\restr[2]{{
  \left.\kern-\nulldelimiterspace 
  #1 
  \vphantom{\big|} 
  \right|_{#2} 
  }}

\newtheorem{theorem}{Theorem}[section]
\newtheorem*{theorem*}{Theorem}
\newtheorem{proposition}[theorem]{Proposition}
\newtheorem*{proposition*}{Proposition}
\newtheorem{lemma}[theorem]{Lemma}
\newtheorem*{lemma*}{Lemma}

\newtheorem*{corollary*}{Corollary}

\newtheorem*{observation*}{Observation}
\newtheorem{definition}[theorem]{Definition}
\newtheorem*{definition*}{Definition}

\newtheorem*{remark*}{Remark}

\begin{document}
\title{Toric surfaces with equivariant Kazhdan-Lusztig atlases}
\author{Bal\'{a}zs Elek}
\date{\today}
\maketitle

\begin{abstract}
A Kazhdan-Lusztig atlas, introduced by He, Knutson and Lu in \cite{HKL}, on a stratified variety $(V,\cY)$ is a way of modeling the stratification $\cY$ of $V$ locally on the stratification of Kazhdan-Lusztig varieties $X^{w}_o\cap X_{v}$ by their intersection with opposite Schubert varieties $X_u$. We are interested in classifying smooth toric surfaces with Kazhdan-Lusztig atlases. This involves finding a degeneration of $V$ to a union of Richardson varieties in the flag variety $H/B_H$ of some Kac-Moody group $H$. We determine which toric surfaces have a chance at having a Kazhdan-Lusztig atlas by looking at their moment polytopes, then describe a way to find a suitable group $H$. More precisely, we find that (up to equivalence) there are $19$ or $20$ broken toric surfaces admitting simply-laced atlases, and that there are at most $7543$ broken toric surfaces where $H$ is any Kac-Moody group. 
\end{abstract}
\tableofcontents

\section{Introduction}\label{sec:introduction}
\begin{definition}\label{def:stratification}
Let $M$ be a variety. By a \textbf{stratification} $\cY_o$ of $M$, we mean a family of locally closed subvarieties indexed by a poset $\cY$ such that: $M=\bigsqcup_{X_o\in Y_o}X_o$, $\overline{X_o}=\bigsqcup_{X_o^\prime\leq X_o} X_o^\prime.$
\end{definition}
\noindent Frequently we'll work with the closures of the pieces, called \emph{strata}, and we'll indicate this by writing $\cY$ instead of $\cY_o$.

\begin{definition}\label{def:bruhatlas}
(He-Knutson-Lu \cite{HKL}) Let $M$ be a manifold with a stratification $\cY$ whose minimal strata are points. A \textbf{Bruhat atlas} on $(M,\cY)$ is the following data:
\begin{enumerate}
\item A Kac-Moody group $H$ with Borel subgroup $B_H$.
\item An open cover for $M$ consisting of open sets $U_f$ around the minimal strata $M=\bigcup_{f\in \cY_{\text{min}}} U_f.$
\item A ranked poset injection $w:\cY^{\text{opp}}\hookrightarrow W_H$ whose image is a union $\bigcup_{f\in \cY_{\text{min}}}[e,w(f)]$ of Bruhat intervals.
\item For $f\in \cY_{\text{min}}$, a stratified isomorphism
\[
c_f:U_f \stackrel{\sim}{\to} X_o^{w(f)} \subset H/B_H, \quad \text{where}\quad X_o^{w(f)}=\overline{Bw(f)B/B}.
\]
\end{enumerate}
\end{definition}
\noindent Examples of manifolds with Bruhat atlases include:
\begin{enumerate}
\item Grassmannians $Gr(k,n)$ with their positroid stratification, whose $H=\widehat{SL(n)}$ (Snider \cite{S}).
\item More generally, partial flag varieties $G/P$ with the stratification by projected Richardson varieties (for this stratification, see Knutson-Lam-Speyer \cite{KLS}) (He-Knutson-Lu \cite{HKL}).
\item Wonderful compactifications of groups (He-Knutson-Lu \cite{HKL}).
\end{enumerate}

\begin{definition}\label{def:eqvtbruhatlas}
Let $(M,\cY)$ be a stratified manifold with an action of a torus $T_M$. An \textbf{equivariant Bruhat atlas} is a Bruhat atlas $(H,\{ c_f \}_{f\in\cY_{\text{min}}},w)$ and a map $T_M\hookrightarrow T_H$ such that
\begin{enumerate}
\item each of the chart maps $c_f$ is $T_M$-equivariant, and
\item there is a $T_M$-equivariant degeneration
\begin{align}
M\rightsquigarrow M^\prime &:= \bigcup_{f\in \cY_{\text{min}}} X^{w(f)} \label{eqn:degen}
\end{align}
of $M$ into a union of Schubert varieties, carrying the anticanonical line bundle on $M$ to the $\cO(\rho)$ line bundle restricted from $H/B_H$.
\end{enumerate}
\end{definition}
\noindent When $M$ is a toric variety (as it will be in this paper) then (\ref{eqn:degen}) gives us a decomposition of $M$'s moment polytope into the moment polytopes of the $X^{w(f)}$'s, e.g.:

\begin{figure}[h]
\centering
\includegraphics{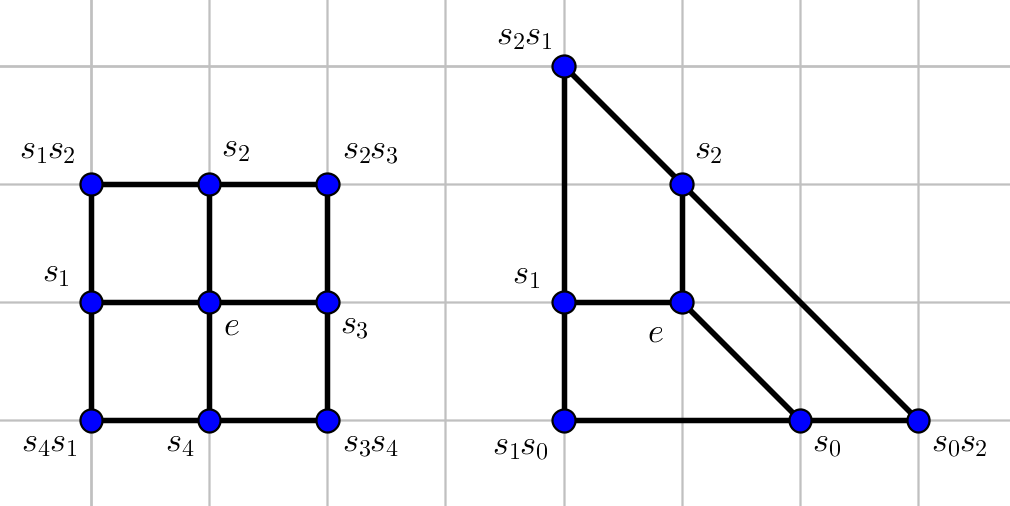}\caption{Equivariant Bruhat atlases}\label{fig:eqvtbruhatlases}
\end{figure}

\noindent The first polytope in Figure \ref{fig:eqvtbruhatlases} (the moment polytope of $\bC\bP^1\times \bC\bP^1$) is subdivided into four smaller squares, which represent $\bC\bP^1\times \bC\bP^1$ degenerating into a union of four $\bC\bP^1\times \bC\bP^1$'s. In the second polytope, $\bC\bP^2$ is degenerating to a union of three Schubert varieties, each isomorphic to the first Hirzebruch surface. The labels of the vertices are coming from the map $w$ of definition \ref{def:bruhatlas}, and the groups $H$ are $(SL_2(\bC))^4$ and $\widehat{SL_2(\bC)}$, respectively.

\subsection{Kazhdan-Lusztig atlases}\label{sec:KLatlas}
For inductive classification purposes, we want to determine what sort of structure a stratum $Z\in\cY$ inherits from the Bruhat atlas on $M$. Each $Z$ has a stratification $\restr{\cY}{Z}$, and an open cover
\[
\bigcup_{f\in \cY_{\text{min}}} U_f\cap Z,\quad\text{with}\quad U_f\cap Z \cong X^{w(f)}_o\cap X_{w(Z)}
\]
compatible with the stratification, since by (\ref{def:bruhatlas}), the isomorphism $U_f\cong X^{w(f)}_o$ is stratified. Therefore $Z$ has an ``atlas'' composed of Kazhdan-Lusztig varieties (defined as $X^{w}_{v,o}=X^{w}_o\cap X_{v}$). This leads us to the following definition:
\begin{definition}\label{def:KLatlas}(He-Knutson-Lu \cite{HKL}) A \textbf{Kazhdan-Lusztig atlas} on a stratified $T$-variety $(V,\cY)$ with $V^T$ finite is:
\begin{enumerate}
\item A Kac-Moody group $H$.
\item A ranked poset injection $w_M:\cY^{\text{opp}}\to W_H$ whose image is $\bigcup_{f\in V^T} [w(V),w(f)]$.
\item An open cover $V=\bigcup U_f$ consisting, around each $f\in V^T$ of an affine variety $U_f$ and a choice of a $T$-equivariant stratified isomorphism
\[
U_f \cong X^{w(f)}_o\cap X_{w(V)}.
\]
In particular, $V$ and $U_f$ need not be smooth.
\item A $T_V$-equivariant degeneration $V \rightsquigarrow V^\prime = \bigcup_{f\in V^T} X^{w(f)}\cap X_{w(V)}\label{eqn:kldegen}$.
\end{enumerate}
\end{definition}

\subsection{Toric surfaces with Bruhat atlases}
We are interested in the classification of manifolds with equivariant Bruhat atlases. We consider toric manifolds as a starting point. Putting an equivariant Bruhat atlas on a toric manifold $M$ would mean associating an element $w(f)\in W_H$ to each face of $M$'s moment polytope (provided we figure out what the group $H$ should be). Obviously there are restrictions to this; for instance, each of the vertex labels must have length equal to $n=\text{dim}(M)$.

The simplest nontrivial case of a toric manifold is a toric surface, and in this case, the moment polytope of $M$ is just a convex polygon.
\begin{theorem}
The only toric surfaces admitting equivariant Bruhat atlases are $\bC\bP^2$ and $\bC\bP^1\times \bC\bP^1$, as in Figure \ref{fig:eqvtbruhatlases}.
\end{theorem}
\begin{proof}
This will follow from our main theorem (see section \ref{sec:listofatlases}).
\end{proof}

\noindent As we mentioned before, our strategy is to to try to classify smooth toric surfaces admitting a Kazhdan-Lusztig atlas, and use this knowledge to answer questions about Bruhat atlases on higher-dimensional manifolds. Our main results are:

\begin{itemize}
\item The classification (in \ref{sec:pieces}) of \textbf{Richardson quadrilaterals}, the moment polytopes of $2$-dimensional Richardson varieties in Kac-Moody flag manifolds with respect to their $\mathcal{O}(\rho)$ line bundles.
\item The classification of all lattice polygons with decompositions into the moment polytopes of Richardsons appearing in simply laced Kac-Moody groups.
\item In the simply laced case: whenever possible, a description of a Kazhdan-Lusztig atlas on each of the smooth toric varieties with lattice polygons as above.
\item In the simply laced case, embeddings of the degenerations (\ref{eqn:kldegen}) in $H/B_H$ for atlases with $H$ of finite type.
\end{itemize}

\section{Pizzas}
\subsection{Motivation and definition}
Since any smooth lattice polygon in $\bZ^2\subset\bR^2$ has a smooth toric variety associated to it, we only need to look at which varieties degenerate to our desired unions of Schubert varieties in various flag manifolds. Since the degeneration preserves symplectic volume and is $T_M$-equivariant, the moment polytope of $M^\prime$ will be a subdivision of that of $M$. Moreover, the newly formed pieces have to be moment polytopes of Richardson surfaces in $H/B_H$, hence they have to be quadrilaterals (since height $2$ intervals in Bruhat order are diamonds). So the moment polytope $\Phi^\prime(M^\prime)$ will look like a sliced up pizza, e.g.

\begin{figure}[H]
\centering
\includegraphics[width=5cm]{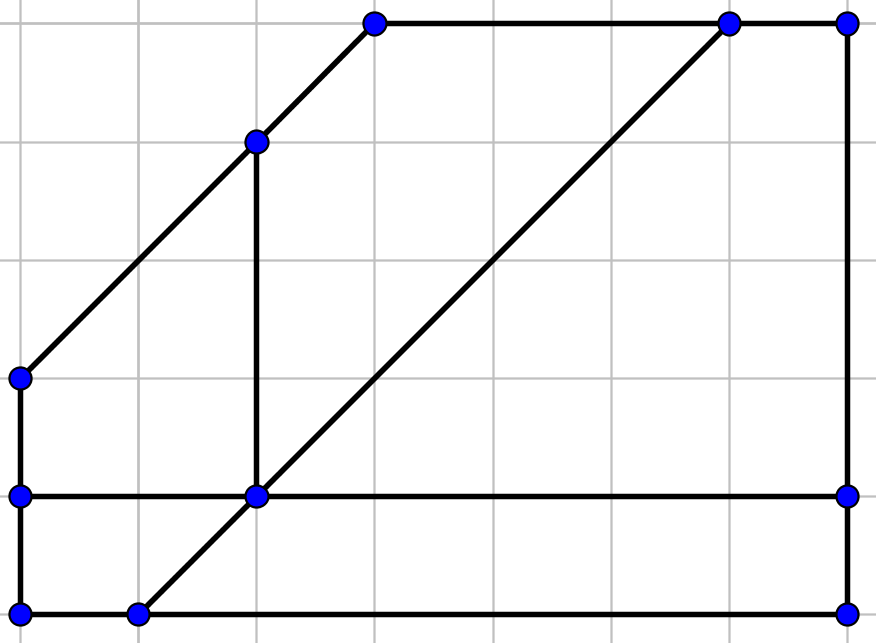}
\end{figure}

Motivated by the above figure, we define
\begin{definition}\label{def:latticepizza} A \textbf{lattice pizza} is a lattice polygon with a ``star-shaped'' subdivision into Richardson quadrilaterals, which will be referred to as \textbf{pizza pieces}\footnote{It is regrettable that we have to avoid the obvious name ``pizza slice'' for these, but slicing already has a standard meaning in mathematics.} (listed in section \ref{sec:pieces}).
\end{definition}
\begin{definition}\label{def:pizza} A \textbf{pizza} is an equivalence class of lattice pizzas under the following equivalence relation: Two lattice pizzas are equivalent if there is a stratification-preserving homeomorphism such that, up to a global $GL(2,\bZ)$-transformation, the angles between the edges match simultaneously.
\end{definition}

We want to see when we can glue a list of pieces into a pizza. If we $SL(2,\mathbb{Z})$-shear a piece to be in a position where the center of the pizza is the piece's bottom right corner, and the edges adjacent to the bottom left corner of the piece are in the position of the standard basis in $\bR^2$ (we will refer to this as the \textbf{standard position}), and compare this to how the next piece has to be glued on, we can associate a matrix in $GL(2,\bR)_+$ to a piece. Consider the following picture of a piece corresponding to an opposite Schubert surface in $A_2$.

\begin{figure}[H]
\centering
\includegraphics[width=3cm]{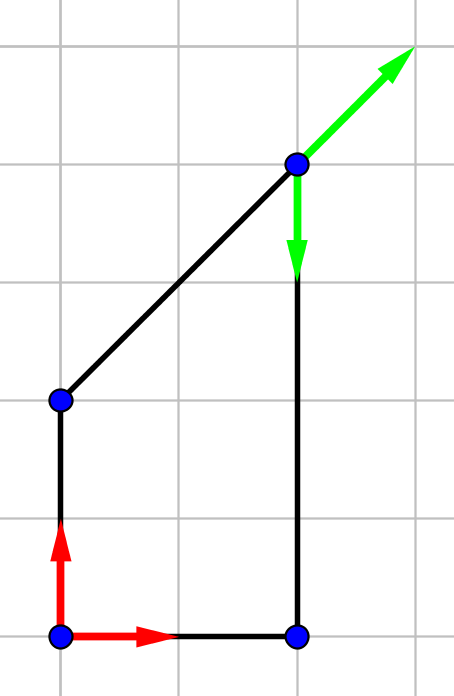}
\end{figure}

\noindent The piece has been sheared to this standard position, with the red basis at the SW corner being the standard basis, and the green basis at the NE corner is where we have to glue the next piece. So we associate the matrix
\[
M=
\begin{pmatrix}
0 & 1 \\
-1 & 1
\end{pmatrix}
\]
to this piece. If the next piece we attach is a $\bC\bP^1\times \bC\bP^1$, then this will change the green basis to the purple one in the following picture:

\begin{figure}[H]
\centering
\includegraphics[width=4cm]{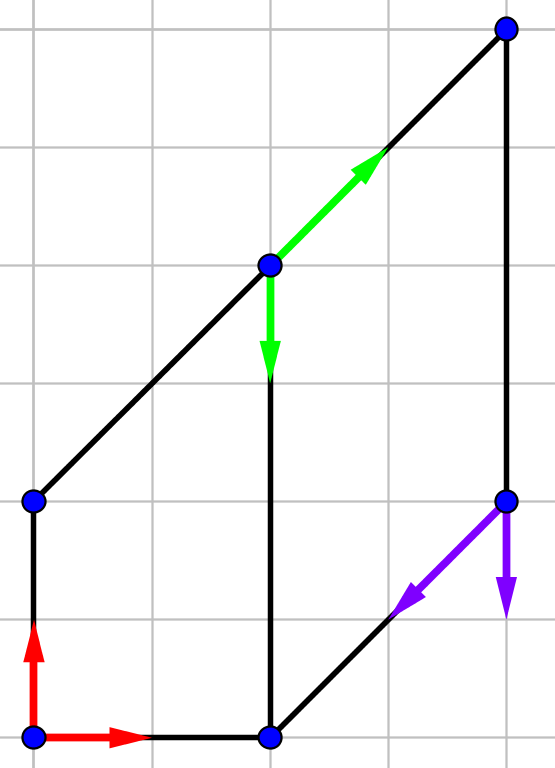}
\end{figure}

In order to find how the basis has changed from the red one to the purple one, note that first we changed the red basis to the green one using $M$, and then used the second piece to turn the green one to the purple one. So if we know how the $\bC\bP^1\times \bC\bP^1$ piece changes the standard basis, say, by a matrix $N$, then we can compute how the red basis turns into the purple one by computing the product
\[
( M N M^{-1} ) M = MN.
\]
So we only need to associate one matrix in $GL_2(\bR)_+$ to a piece, namely the one where the bottom right corner has been moved to the standard basis position. It is not hard to see that the matrix associated to $\bC\bP^1\times \bC\bP^1$ is
\[
N=
\begin{pmatrix}
0 & 1 \\
-1 & 0 
\end{pmatrix}
\]
and we obtain
\[
MN=\begin{pmatrix}
0 & 1 \\
-1 & 1
\end{pmatrix}
\begin{pmatrix}
0 & 1 \\
-1 & 0 
\end{pmatrix}
=\begin{pmatrix}
-1 & 0\\
-1 & -1
\end{pmatrix}
\]
which indeed corresponds to the purple basis. Now it should be easy to believe the following theorem:
\begin{theorem}\label{thm:pizzacondition}
Let $M_1,M_2,\ldots, M_l$ be the matrices associated to a given list of pizza pieces. If the pieces form a pizza, then $\prod_{i=1}^l M_i = 
\begin{pmatrix}
1 & 0 \\
0 & 1
\end{pmatrix}$.
\end{theorem}
Since our pizza pieces are all lattice polygons, we know that the $GL(2,\bR)_+$-matrices associated to the pizza pieces will have integer entries. Therefore, in order to satisfy Theorem \ref{thm:pizzacondition}, all the matrices of the pieces will in fact be in $SL(2,\bZ)$.

The above condition is necessary, but we can make some further observations to reduce this to a finite problem. We would like to embed our pizza in $\bR^2$, so we would like to wind around the origin once using the pieces. To contend with the winding number, we lift these matrices from $SL(2,\bR)$ to its universal cover $\widetilde{SL_2(\bR)}$. We will represent an element of $\widetilde{SL_2(\bR)}$ by its matrix $M$, together with a homotopy class of a path $\gamma$ in $\bR^2\setminus \overrightarrow{0}$ connecting $\begin{pmatrix} 1 \\ 0 \end{pmatrix}$ to $M\begin{pmatrix} 1 \\ 0 \end{pmatrix}$. Elements of $\widetilde{SL_2(\bR)}$ multiply by multiplying the matrices and concatenating the paths appropriately.

We will therefore, associate to a pizza piece a pair $(M,\gamma)$ where $M$ is the matrix defined above and $\gamma$ is the (class of the) straight line path connecting $\begin{pmatrix} 1 \\ 0 \end{pmatrix}$ to $M\begin{pmatrix} 1 \\ 0 \end{pmatrix}$, i.e.
\begin{figure}[H]
\centering
\includegraphics{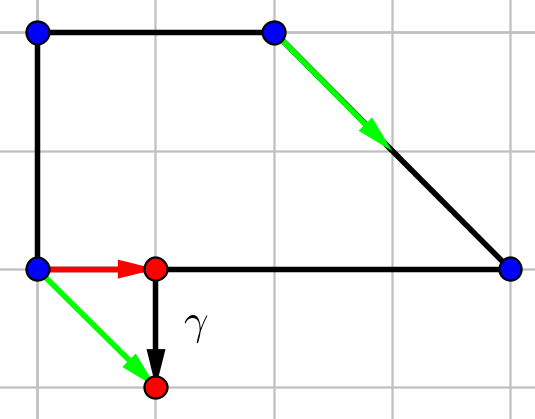}
\end{figure}
\noindent By direct check on the list of pieces (in section \ref{sec:pieces}), we note that none of these straight line paths pass through the origin. Then attaching a pizza piece $(N,\mu)$ clockwise to a sequence of pieces with current basis $M$ and current path $\gamma$ will yield $(MN, \gamma \circ M(\mu))$. Consequently, if a given set of pieces results in a pizza, we will have a closed loop around the origin based at $\begin{pmatrix} 1 \\ 0 \end{pmatrix}$, with a well-defined winding number $1$. Also, as this path is equivalent (by sending all vectors to their negatives) to the path consisting of following the primitive vectors of the spokes (\ref{def:spoke}) of the pizza, the winding number will coincide with the number of layers of our pizza, as exemplified in the following picture:
\begin{figure}[H]
\centering
\includegraphics{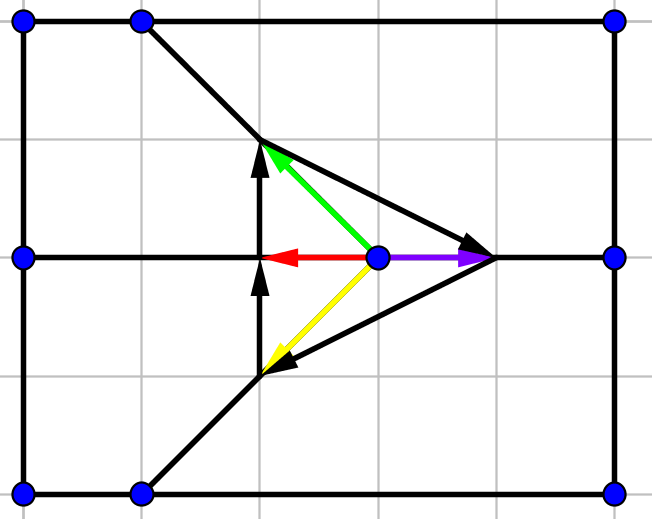}
\end{figure}

\begin{theorem}\label{thm:braidgroup} (Wikipedia) The preimage of $SL_2(\bZ)$ inside $\widetilde{SL_2(\bR)}$ is $Br_3$, the braid group on 3 strands.
\end{theorem}

\noindent A pizza piece therefore could be associated an element of $Br_3$, but for practical reasons we would prefer to work with matrices instead of braids.

We will represent braids in terms of the standard braid generators in Figure \ref{fig:braidgens}.

\begin{figure}[H]
\centering
\includegraphics{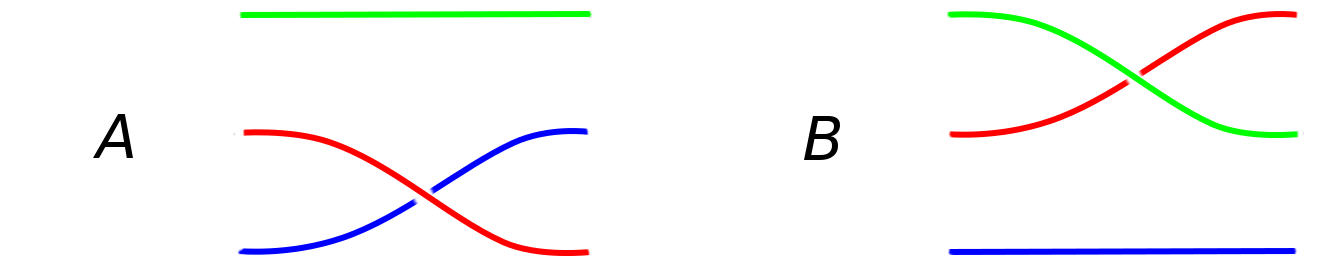}
\caption{Generators of the braid group\label{fig:braidgens}}
\end{figure}

\begin{lemma}\label{lem:sl2br3} The map $Br_3 \rightarrow SL_2(\bZ)\times \bZ$, with second factor $\mathrm{ab}$ given by abelianization, is injective.
\end{lemma}
\begin{proof}
The kernel of the map $Br_3\to SL_2(\bZ)$ is generated by the ``double full twist'' braid $(AB)^6$, while $\mathrm{ab}$ sends both generators to $1$, so $\mathrm{ab}((AB)^6)=12$.
\end{proof}

It was easy to determine the $SL(2,\bZ)$-matrix of a piece by just looking at it, but determining $\mathrm{ab}(S)$ is a little more subtle. Since abelianization is a functor, the following Lemma gives us some clues:

\begin{lemma}\label{lem:sl2zmod12}
(Example 2.5. in \cite{Ko}) The abelianization of $SL_2(\bZ)$ is $\bZ/12\bZ$. Moreover, for
\[
\begin{pmatrix}
a & b \\
c & d
\end{pmatrix}\in SL_2(\bZ),
\]
the image in $\bZ/12\bZ$ can be computed by taking
\[
\chi \begin{pmatrix}
a & b \\
c & d
\end{pmatrix}
= ((1-c^2 )(bd+3(c-1)d+c+3)+c(a+d-3))/12\bZ.
\]
\end{lemma}

So from the matrix of a piece $S$, we can determine $\mathrm{ab}(S)\mod 12$. To figure out the exact value, we notice that if one can build a pizza from the given sequence of pieces, then we must have $\sum_S\mathrm{ab}(S)\equiv 0\mod 12$. If one further insists that the pizza should be ``single-layered'', then we must have $\sum_S\mathrm{ab}(S)=12$. So for instance, the existence of the two pizzas in Figure \ref{fig:eqvtbruhatlases} implies that $\mathrm{ab}(\bP^1\times \bP^1)=3$ and $\mathrm{ab}(X^{s_1s_2})=4$ where $X^{s_1s_2}\in SL_2{\bC}$ is a Schubert variety. Then we can use the list of pizzas (section \ref{sec:listofpizzas}) to figure out the values of the other pieces.

\begin{definition}\label{nutritive} For a piece $S\in Br_3$, define the \textbf{nutritive value} $\nu(S)$ of $S$ as the rational number $\frac{m}{12}$ where $m=\mathrm{ab}(S)$.
\end{definition}

Now we can make sure that our pizza is bakeable in a conventional oven by requiring that $\sum_S\nu(S)=\frac{12}{12}$. This (almost) reduces this part of the classification to a finite problem.

\subsection{Pizza pieces}\label{sec:pieces}
It follows from Definition \ref{def:KLatlas} that the pieces of the pizza (c.f. Definition \ref{def:pizza}) must be moment polytopes of Richardson surfaces in $H$. We will use the shorthand $X_v^w=X_v\cap X^w$ for Richardson varieties. To obtain a classification, we would like to list all the isomorphism types of moment polytopes of Richardson surfaces in arbitrary Kac-Moody groups. We will need the following strengthening of a special case of Corollary 3.11. of \cite{D}:
\begin{proposition}\label{thm:richardsonquads}
The moment polytope of a Richardson surface in any $H$ is part of the X-ray of the moment polytope (with possibly not the $V(\rho)$-embedding) of a flag manifold of a rank $2$ Kac-Moody group.
\end{proposition}
\begin{proof}
Let $X^w_v$ be a Richardson surface in $H$. We know that $v\lessdot r_\alpha v \lessdot r_\beta r_\alpha v=w$ for some positive roots $\alpha, \beta$. The moment polytope of $X^w_v$ is a quadrilateral with edge labels:
\begin{figure}[H]
\centering
\includegraphics{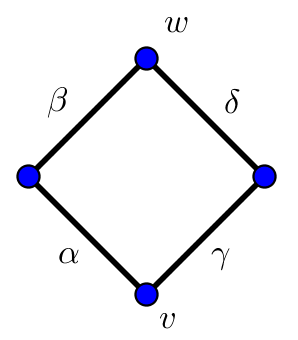}
\end{figure}
\noindent We claim that $\gamma, \delta\in \mathrm{Span}_\bR (\alpha, \beta)$. Since the polytope is $2$-dimensional, and $\{ v(\alpha), v(\beta) \}$ is linearly independent, we know that $v(\gamma), v(\delta) \in \mathrm{Span}_\bR(v(\alpha), v(\beta))$, and $v$ is a linear transformation. Therefore all roots that are labeling the edges of this quadrilateral lie in a $2$-dimensional subspace of $\fh^*$, so if we intersect the root system of $H$ with the $2$-plane $\mathrm{Span}_\bR (\alpha, \beta)$, we obtain a rank $2$ root system with corresponding Kac-Moody group $H^\prime=Z_H(\ker \alpha \cap \ker \beta)$. Then, up to the equivalence relation in definition \ref{def:pizza}, $X^w_v$'s polytope will appear in $H^\prime/B_{H^\prime}$. 
\end{proof}

It remains to look for moment quadrilaterals in all rank $2$ Kac-Moody groups. The (bottom of the) moment polytope of $\widetilde{A_1}$ is:
\begin{figure}[H]
\centering
\includegraphics{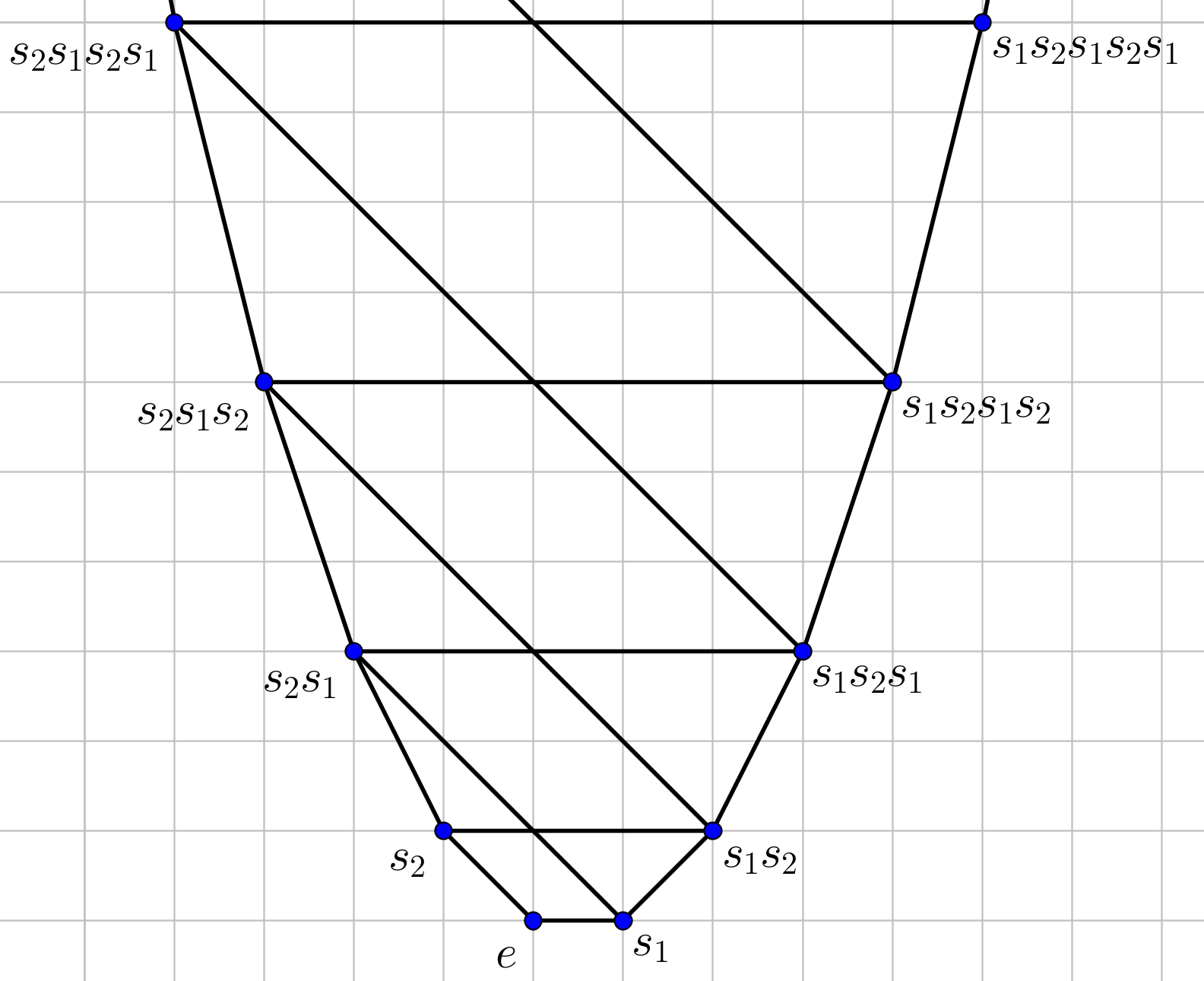}
\end{figure}
There are only a couple of types of quadrilaterals to check here:
\begin{enumerate}
\item The two Schubert surfaces $X_e^{s_1s_2}$ and $X_e^{s_2s_1}$ are smooth, and they appear in $B_2$.
\item Those of the form $X_{s_1v}^{s_1w}$ or $X_{s_2v}^{s_1w}$ are all singular, as the primitive vectors from the top right vertex are $\begin{pmatrix} -1 \\ 0 \end{pmatrix}$ and $\begin{pmatrix} -1 \\ -k \end{pmatrix}$ for $k\geq 2$.
\item Those of the form $X_{s_2v}^{s_2w}$ or $X_{s_1v}^{s_2w}$  are all singular, as their top left vertex will have primitive vectors $\begin{pmatrix} 1 \\ -1 \end{pmatrix}$ and $\begin{pmatrix} 1 \\ -k \end{pmatrix}$ for $k\geq 3$.
\end{enumerate}

A similar situation arises in the Kac-Moody groups arising from the generalized Cartan matrix $\begin{pmatrix} 2 & -1 \\ -k & 2 \end{pmatrix}$, and, more generally $\begin{pmatrix} 2 & -j \\ -i & 2 \end{pmatrix}$. The only smooth Richardson surfaces are the Schubert surfaces, and they are of the form (either the red or the yellow vertex must be in the center):
\begin{figure}[H]
\caption{$KM(k)$}
\centering
\includegraphics{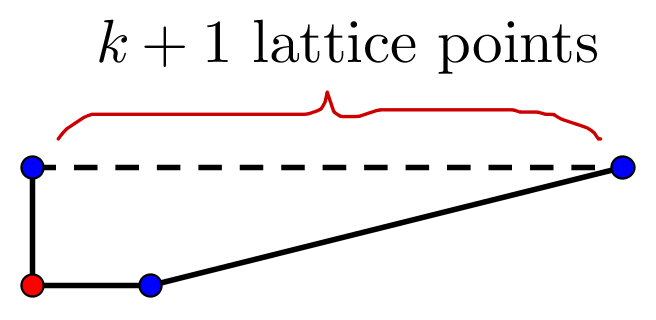}
\end{figure}

To find the nutritive value of $KM(k)$ for $k\geq 4$, note the difference between the pieces $KM(k)$ and $KM(k+1)$:
\begin{figure}[H]
\centering
\includegraphics{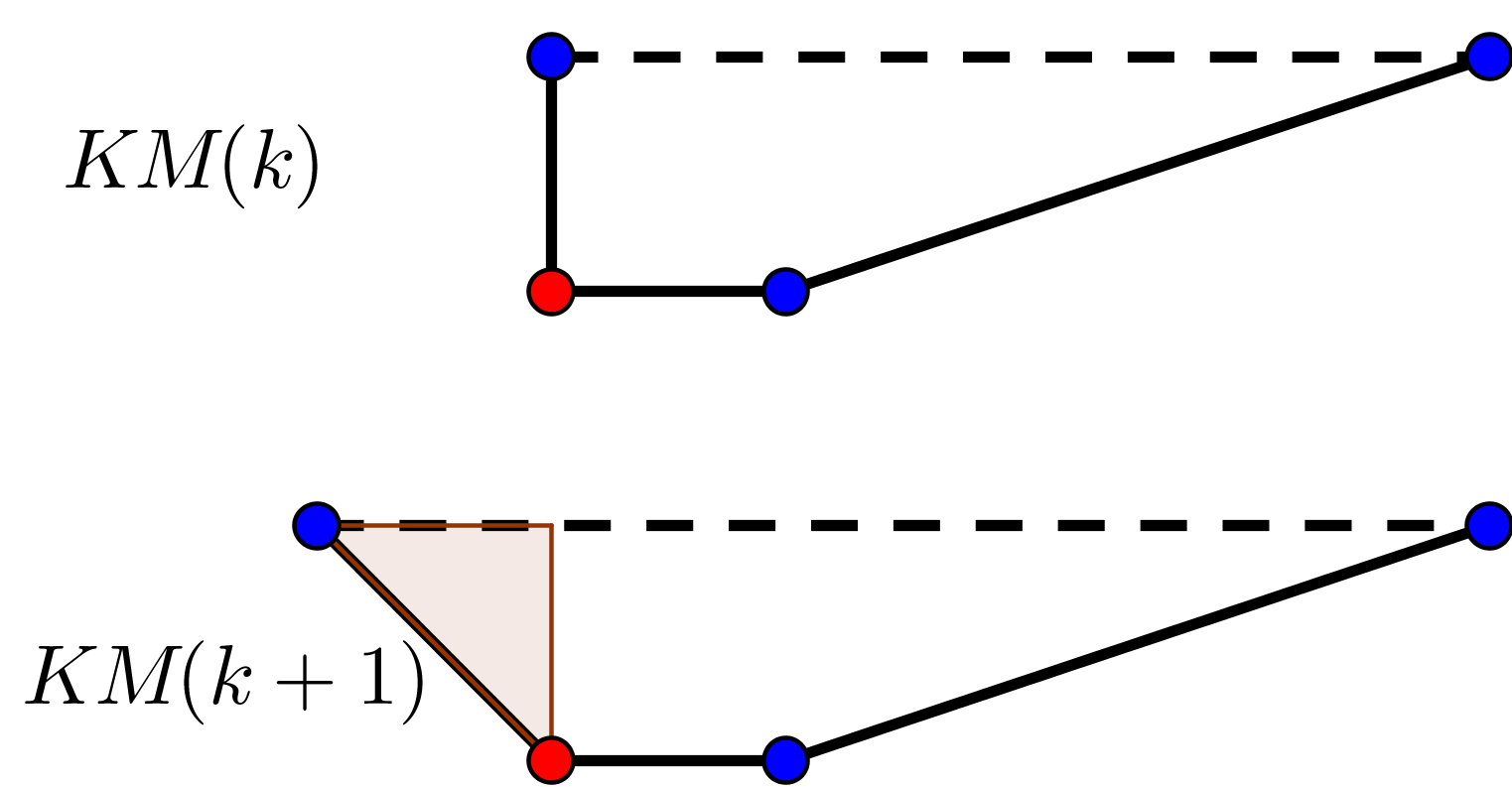}
\end{figure}
\begin{lemma}\label{lem:KMk} The correct lift to the braid group of the pizza piece $KM(k)$ is $B^kABA$ in terms of the standard braid generators in Figure \ref{fig:braidgens}.
\end{lemma}
\begin{proof}
We will show this by showing that the pizza piece $KM(k)$ (as a braid) is equivalent to a sequence of pieces, whose lifts we already know. Let $S(k)=A_2^{opp}b,B_2^{opp}b,\ldots ,B_2^{opp}b$ (with $k-1$ $B_2^{opp}b$s). We claim that the piece $KM(k)$ is equivalent to the sequence of pieces $S(k),A_2^{opp}b$. We will induct on $k$. Since $KM(1)$ is the Schubert variety in $A_2$'s flag manifold whose lift is $BABA$, and the lift of $A_2^{opp}b$ is $BA$, the base case holds. We may slice the piece $KM(k+1)$ as the following picture suggests:
\begin{figure}[H]
\centering
\includegraphics[width=\textwidth]{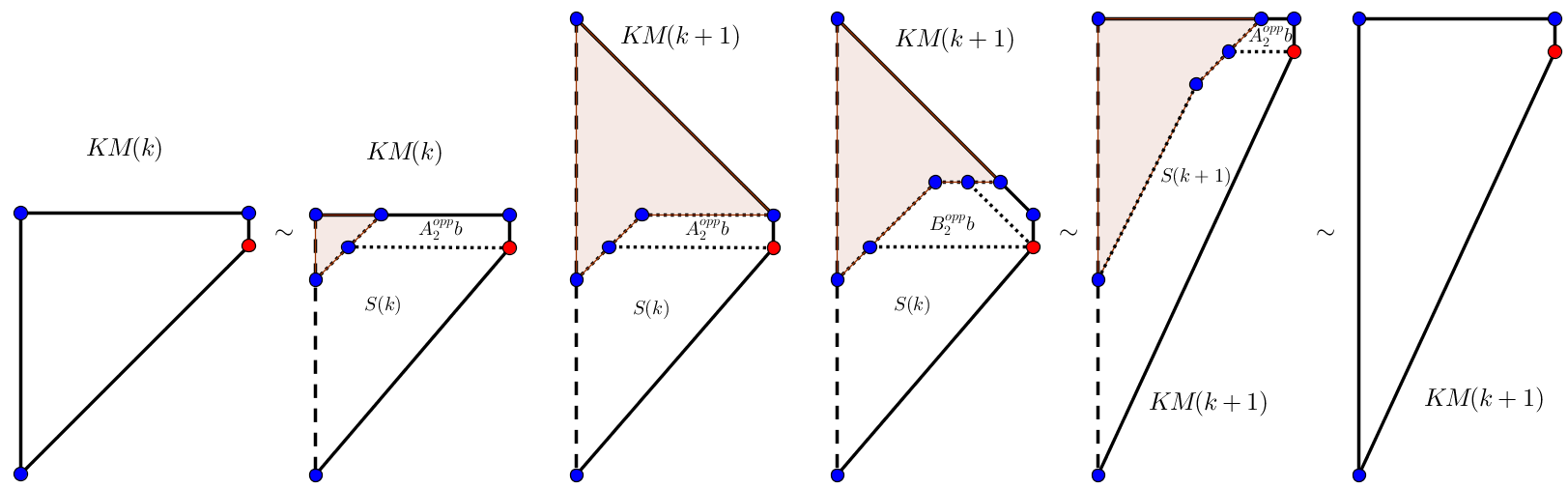}
\end{figure}
Note that the shaded region is the only difference between $KM(k+1)$ and the sequence $S(k),B_2^{opp}b,A_2^{opp}b$, and it is irrelevant to how this piece, or the sequence of pieces fits into a pizza. In terms of braids, this means that in $Br_3$, $KM(k+1)$ lifts to the same element as $S(k),B_2^{opp}b,A_2^{opp}b=A_2^{opp}b, (B_2^{opp}b)^{k}) ,A_2^{opp}b$, which is
\[
BA*(BAB^{-1})^k*BA=B^kABA.
\]
Note that this implies that for $k\geq 1$, $\nu(KM(k))=\frac{k+3}{12}$, in particular, the pieces $KM(k)$ for $k\geq 10$ can never be part of a pizza for nutritional reasons.
\end{proof}

Therefore since we want $M$ smooth, we may start with rank $2$ finite type groups, and look at all the equivalence classes of polytopes of Richardson surfaces there, including the infinite family above, then add $KM(k)$ for $k=4,\ldots, 9$ to the list ($KM(10)$ is more nutritious than a whole pizza). Below we give a table of the Richardson quadrilaterals together with the corresponding matrices in $SL(2,\bZ)$. Note that if a piece has matrix $A$ then the piece backwards (i.e. reflected across the $y$-axis) has matrix
\[
\begin{pmatrix}
1 & 0 \\
0 & -1
\end{pmatrix}
A^{-1}
\begin{pmatrix}
1 & 0 \\
0 & -1
\end{pmatrix}
\]
In the table the center is always the bottom left vertex (in red). In case of non-simply-laced groups, $s_1$ is always the reflection across the short root. We display the smallest (by edge-length) pieces, but will consider pieces up to equivalence by the equivalence relation in Definition \ref{def:pizza}. We write the braid in terms of the generators in Figure \ref{fig:braidgens}.

\begin{tabular}{ | c | c | c | c | c | m{4cm} |}
\hline
name & Richardson surface & $SL(2,\bZ)$ matrix & braid & $\nu$ & Richardson quadrilateral \\
\hline
$A_1\times A_1$ & $X^{s_1s_2}_e$ & $\begin{pmatrix}
0 & 1 \\
-1 & 0
\end{pmatrix}$ & $ABA$ & $\frac{3}{12}$ & \vspace{0.1cm}\includegraphics{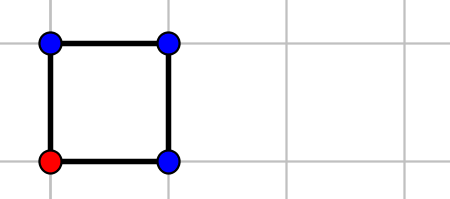} \\
\hline
$A_2$ & $X^{s_1s_2}_e$ & $\begin{pmatrix}
0 & 1 \\
-1 & -1
\end{pmatrix}$ & $BABA$ & $\frac{4}{12}$ & \vspace{0.1cm}\includegraphics{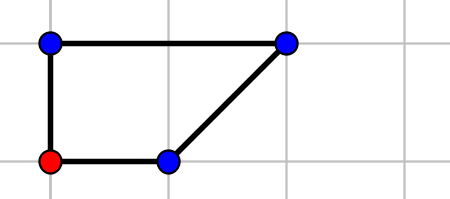} \\
\hline
$A_2^{\text{opp}}$ & $X_{s_1}^{w_0}$ & $\begin{pmatrix}
0 & 1 \\
-1 & 1
\end{pmatrix}$ & $AB$ & $\frac{2}{12}$ & \vspace{0.1cm}\includegraphics{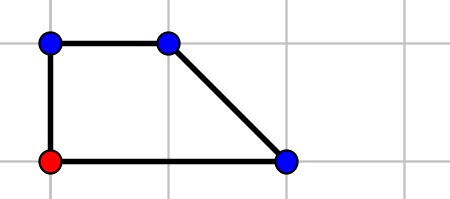} \\
\hline
$B_2$ & $X^{s_1s_2}_e$ & $\begin{pmatrix}
0 & 1 \\
-1 & -2
\end{pmatrix}$ & $BBABA$ & $\frac{5}{12}$ & \vspace{0.1cm}\includegraphics{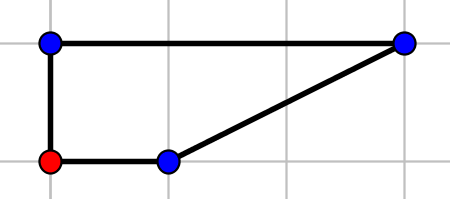} \\
\hline
$B_2^{\text{opp}}$ & $X^{w_0}_{s_1s_2}$ & $\begin{pmatrix}
0 & 1 \\
-1 & 2
\end{pmatrix}$ & $B^{-1}AB$ & $\frac{1}{12}$ & \vspace{0.1cm}\includegraphics{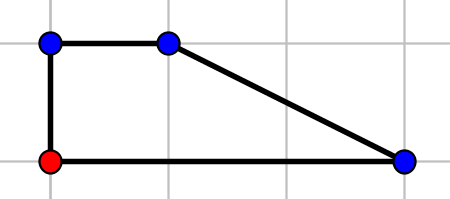} \\
\hline
$B_2^{\text{sing}}$ & $X^{s_2s_1s_2}_{s_2}$ & $\begin{pmatrix}
1 & 1 \\
-2 & -1
\end{pmatrix}$ & $BBA$ & $\frac{3}{12}$ & \vspace{0.1cm}\includegraphics{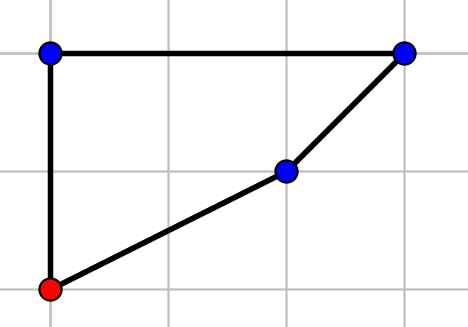} \\
\hline
$G_2$ & $X^{s_1s_2}_e$ & $\begin{pmatrix}
0 & 1 \\
-1 & -3
\end{pmatrix}$ & $BBBABA$ & $\frac{6}{12}$ & \vspace{0.1cm}\includegraphics[width=4cm]{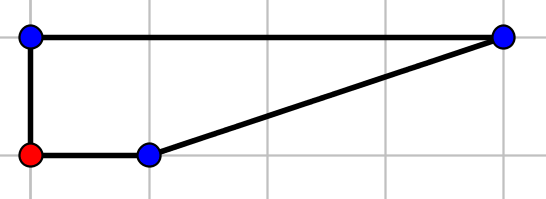} \\
\hline
$G_2^{\text{opp}}$ & $X^{w_0}_{s_1s_2s_1s_2}$ & $\begin{pmatrix}
0 & 1 \\
-1 & 3
\end{pmatrix}$ & $B^{-1}B^{-1}AB$ & $\frac{0}{12}$ & \vspace{0.1cm}\includegraphics[width=4cm]{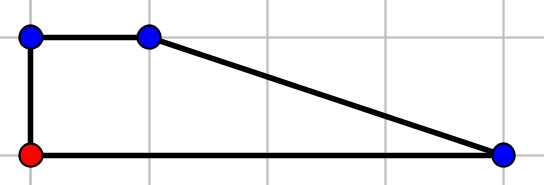} \\
\hline

\end{tabular}

\begin{tabular}{ | c | c | c | c | c | m{5cm} |}
\hline
name & Richardson surface & $SL(2,\bZ)$ matrix & braid & $\nu$ & Richardson quadrilateral \\
\hline
$KM(k)$ & $X^{s_1s_2}_e$ & $\begin{pmatrix}
0 & 1 \\
-1 & -k
\end{pmatrix}$ & $B^kABA$ & $\frac{k+3}{12}$ & \vspace{0.1cm}\includegraphics[width=5cm]{k.png} \\
\hline
$G_2^{\text{short}}$ & $X^{s_2s_1s_2}_{s_2}$ & $\begin{pmatrix}
1 & 1 \\
-3 & -2
\end{pmatrix}$ & $BBBA$ & $\frac{4}{12}$ & \vspace{0.1cm}\includegraphics[width=5cm]{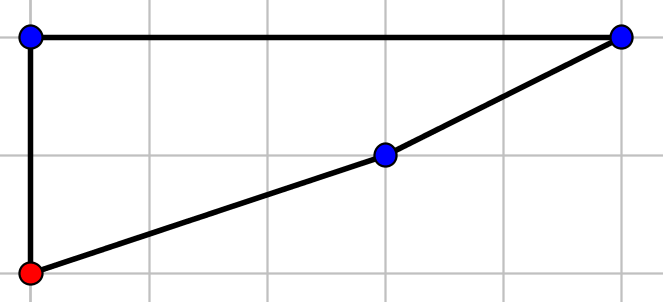} \\
\hline
$G_2^{\text{long}}$ & $X^{s_2s_1s_2s_1s_2}_{s_2s_1s_2}$ & $\begin{pmatrix}
2 & 1 \\
-3 & -1
\end{pmatrix}$ & $BBAB^{-1}$ & $\frac{2}{12}$ & \vspace{0.1cm}\includegraphics[width=5cm]{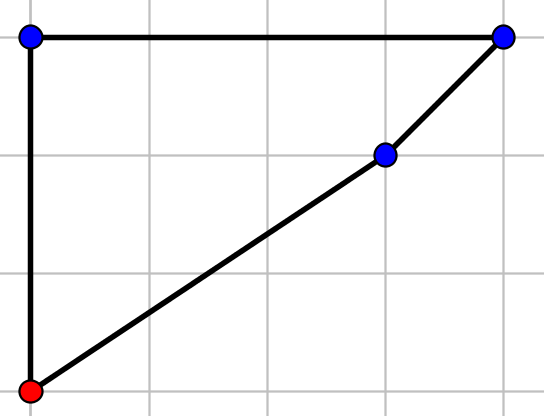} \\
\hline
\end{tabular}

\section{Simply laced pizzas}
Since the $G_2^{\text{opp}}$ piece has nutritive value $\frac{0}{12}$, it could appear arbitrarily many times in a pizza (we will find limits in section \ref{sec:nonsimplylacedpizzas}). To avoid this inconvenience, we restrict our attention to \textbf{simply-laced pizzas}, i.e. pizzas with pieces from simply laced groups only. Note that since $\widetilde{A_2}$ is simply laced and contains a subgroup $\widetilde{A_1}$, we have to include the $B_2$ piece together with the $A_1\times A_1$, $A_2$, and $A_2^{\text{opp}}$ in the list of pieces we are allowed to use.

\subsection{List of simply laced pizzas}\label{sec:listofpizzas}
Since the invariants of all allowed pieces are strictly positive, we just have to list all possible arrangements of the pieces where the nutritive values add up to $1$, and check if the resulting matrices multiply to the identity. The following list of all the $20$ inequivalent pizzas has been obtained by this brute force computation in Sage \cite{sage}:
\begin{figure}[H]
\centering
\includegraphics{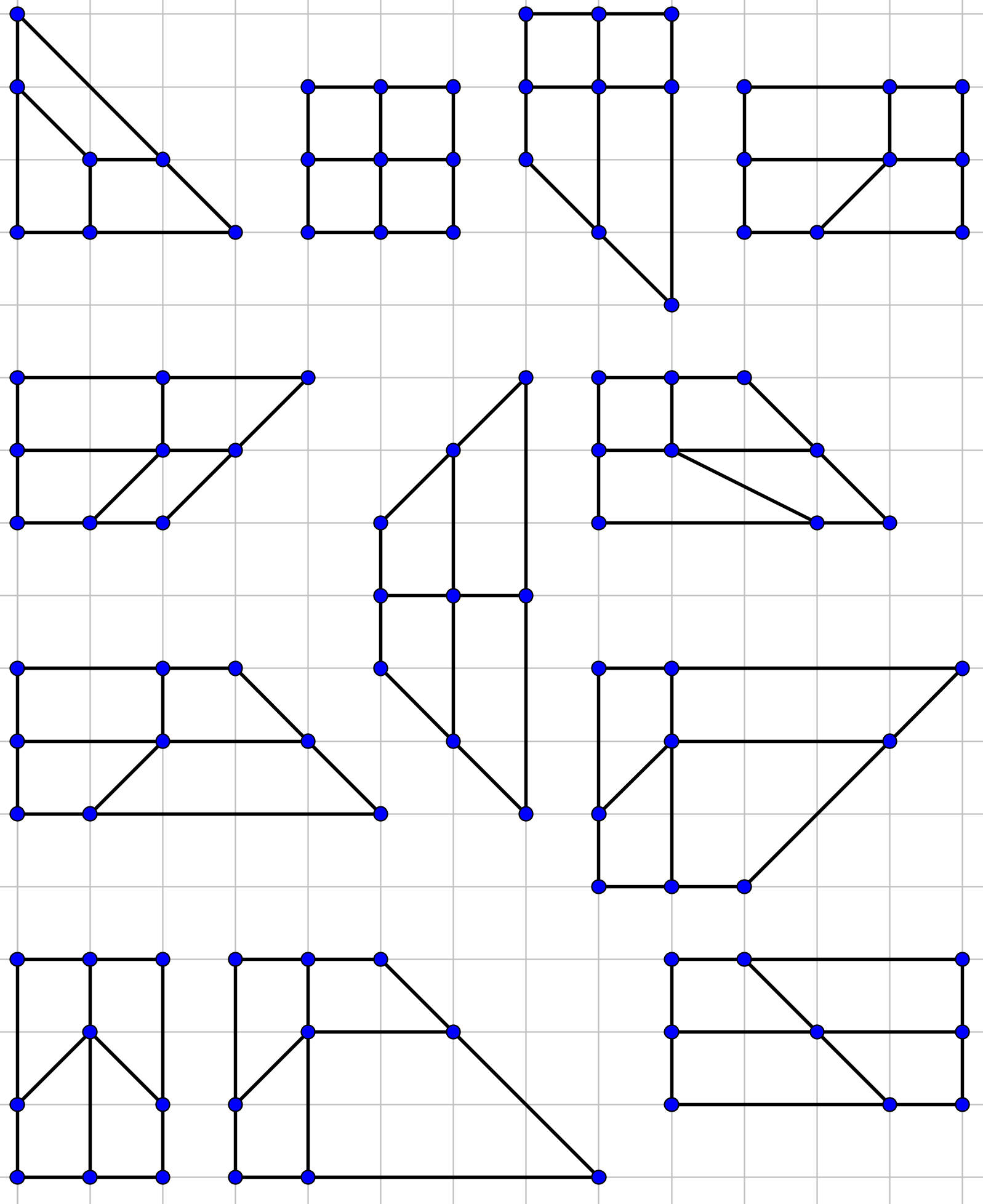}
\end{figure}
\begin{figure}[H]
\centering
\includegraphics{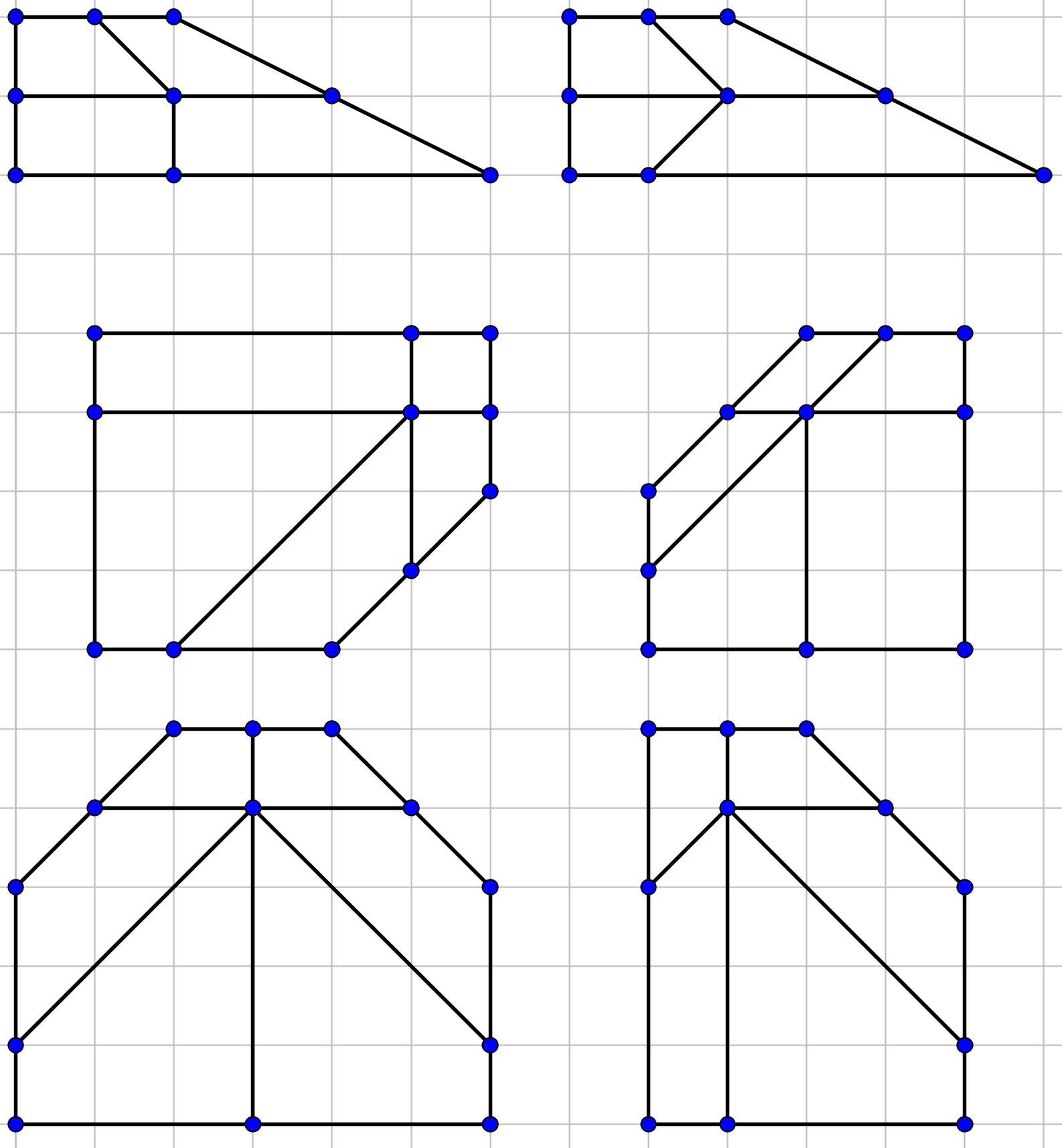}
\end{figure}
\begin{figure}[H]
\centering
\includegraphics{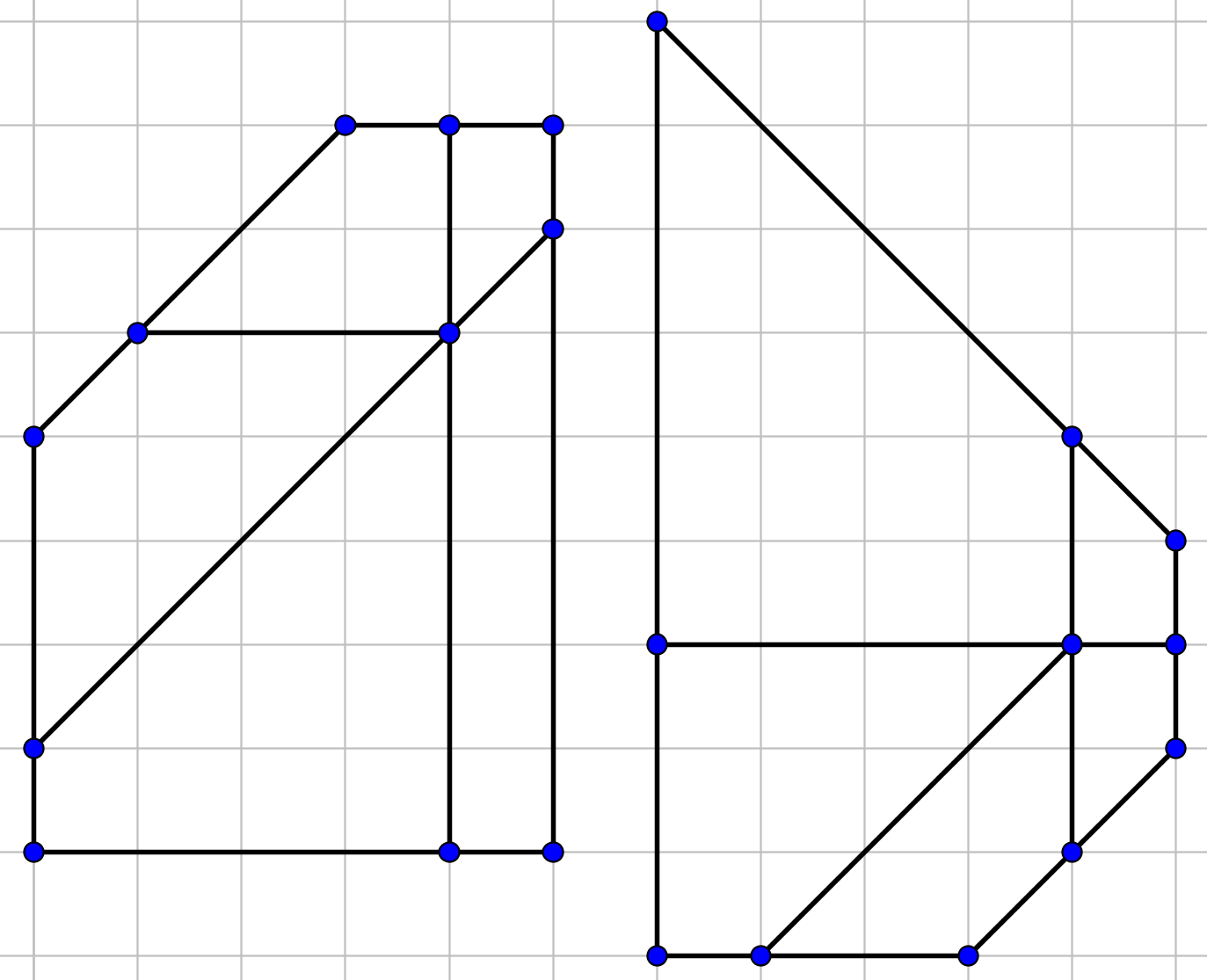}
\end{figure}

\section{Putting atlases on pizzas}\label{sec:atlasesonpizzas}
Everything we did so far was to derive some necessary conditions for part $4$ of Definition \ref{def:KLatlas} to be satisfied. To actually put a K-L atlas on a lattice pizza, we need to specify $H$ and a map $w_M$ from the vertices of the pizza to $W_H$. This is not easy, as in general, there are many choices for such an $H$ and $w_M$. For instance, if one pair $(H,w_M)$ exists, then we can take $H\times H^\prime$ and $w_M\times v$ for some constant $v\in W_{H^\prime}$. We will look for atlases that are minimal in some sense.

The map $w_M$ will label the vertices in a lattice pizza by elements of $W_H$, with the edges corresponding to covering relations in Bruhat order. All covering relations $v\lessdot w$ are of the form $vs_\beta=w$ for some positive root $\beta$ (note that this does not privilege right-multiplication, as equivalently $r_{v\cdot \beta}v=w$). We will label the edges of a pizza by the roots in the covering relations. As we remarked before, we can do this by labeling them by left- or right-multiplication.

\begin{lemma}\label{lem:leftmult} If we label the edges of $H/B_H$'s moment polytope by left-multiplication, then any two edges with identical labels are parallel.
\end{lemma}
\begin{proof}
Let $v_i,w_i$ be elements of $W_H$ labeling vertices of the pizza such that $w_i=r_\beta v_i$. Since $\beta$ is a positive root, there is an associated subgroup $SL_2^\beta\cong SL_2(\bC)$ of $H$. Then the $T$-invariant $\bC\bP^1$'s with fixed points $v_i,w_i$ are the $SL_2^\beta$-orbits of the $v_i$'s. Therefore their moment polytopes (lattice line segments) are parallel to $\beta$.
\end{proof}

So we could label the edges by left-multiplication, but this is slightly redundant.

\begin{lemma}\label{lem:rightmult} If we label the edges of $H/B_H$'s moment polytope by right-multiplication, then the labeling roots are the equivariant homology classes of the corresponding invariant $\bC\bP^1$'s in $H_2^T(H/B_H)$.
\end{lemma}
\begin{proof}
Since labeling by right-multiplication is (left-)$W$-invariant, it suffices to check this for $w=e$, where left and right multiplication are the same, and we get our result by lemma \ref{lem:leftmult}.
\end{proof}

\begin{theorem}\label{thm:rholinebundle} For the labeling by right-multiplication, the lattice length of an edge in a lattice pizza equals the height of the corresponding root.
\end{theorem}
\begin{proof}
An edge in a pizza corresponds to an embedded $T$-invariant $\bP^1$ in $H/B_H$. In general, a $T$-equivariant line bundle over $\bP^1$ is constructed by letting $T$ act on 
\[
\mathcal{O}(\lambda,\mu)=\mathcal{O}(m)\qquad\text{where }\bP^1=\bP(\bC_\lambda \oplus \bC_\mu).
\]
The moment polytope of such a variety is an interval in $\mathfrak{t}^*$ with endpoints $\lambda$ and $\mu$. Let $\frac{\mu-\lambda}{m}$ be the primitive vector in that interval, then we compute
\begin{align*}
\int_{\bP^1}c_1(\mathcal{O}(\lambda,\mu)) &= \int_{\bP^1}c_1(\mathcal{O}(0,\mu-\lambda)) &&=\int_{\bP^1}c_1\left(\mathcal{O}\left(0,\frac{\mu-\lambda}{m}\right)^{\otimes m}\right) \\
&=m\int_{\bP^1}c_1\left(\mathcal{O}\left(0,\frac{\mu-\lambda}{m}\right)\right) &&= m\int_{\bP^1}c_1\left(\mathcal{O}\left(0,(1,0,\ldots , 0)\right) \right)= m,
\end{align*}
which is the number of lattice points in the interval, and we may move the primitive element $\frac{\mu-\lambda}{m}$ to $(1,0,\ldots , 0)$ by applying an element of $SL(n,\bZ)$.

For an arbitrary $H$-weight $\nu$, we have
\[
\xymatrix{
\mathcal{O}(m)\ar[d] \ar[r] & L(\nu=\sum_i c_i\omega_i) \ar[d] \\
\bP^1 \ar@{^{(}->}[r]^i & H/B_H
}
\]
where $i_*([\bP^1])=\sum_i d_i [X^{s_i}]$ in $H_2(H/B_H)$. We also know that the divisor line bundle for the opposite Schubert divisor $X_{s_i}$ is $L(\omega_i)$. So we have
\begin{align*}
m&=\int_{\bP^1} c_1 (i^*(\mathcal{O}(\nu)) \\
&= \int_{\bP^1} i^* \left( c_1(\mathcal{O}(\nu) \right) &\text{by naturality of }c_1\\
&= [c_1(L(\nu))] \cup [i_*(\bP^1)] &\text{by the push-pull formula}\\
&= \left[\sum_i c_i X_{s_i}\right] \cup \left[\sum_i d_i X^{s_i}\right]=\sum_i c_id_i &\text{by duality of the bases }\{X_{s_i}\},\{X^{s_i}\}
\end{align*}
In particular, for $\nu=\rho$, we get $\sum_i c_id_i=\sum_i d_i=\mathrm{ht}(\mu-\lambda)$, since $\mu-\lambda$ is a root by assumption.
\end{proof}

This is promising, since now if an edge in a lattice pizza is length $1$, then it must correspond to a simple root, and if we find enough of them, we might be able to find an $H$ we are looking for. However, the situation is more complicated in general, since it may happen that a certain pizza has no Kazhdan-Lusztig atlas, but a different lattice pizza in the same pizza class does. We will give an example for this in section \ref{sec:toppings}.

Also, we know that length $2$ Bruhat intervals are all diamonds, which leads us to the following Lemma:
\begin{lemma}\label{lem:diamond}
Let $\alpha,\beta$ be roots in some simply laced root system such that $\alpha+\beta$ is a root. Let $w\in W$ and $C=\{ wr_\alpha ,wr_\beta, wr_{\alpha+\beta} \}$. If two elements of $C$ cover $w$ in Bruhat order and are covered by $\widetilde{w}$, then the third element of $C$ cannot cover $w$.
\end{lemma}
\begin{proof}
We will prove the statement for $w_1=wr_\alpha$, $w_2=wr_{\alpha+\beta}$; the other cases are symmetric. Assume that $w\lessdot wr_\beta$. We know that height two Bruhat intervals are diamonds, so it suffices to show that $\widetilde{w}$ covers $wr_\beta$, and we will have a contradiction. By the assumptions on $\alpha$ and $\beta$, we could choose $\alpha$ and $\beta$ to be both be simple roots, so we have
\[
r_{\alpha+\beta}=r_\alpha r_\beta r_\alpha = r_\beta r_\alpha r_\beta.
\]
Choose a reduced word $\widetilde{w}=s_1\cdots s_l$. Then $wr_\alpha=s_1\cdots \widehat{s_i} \cdots s_l$ and $wr_{\alpha+\beta}=s_1\cdots \widehat{s_j} \cdots s_l$. We may assume without loss of generality that $i<j$, so $w=s_1\cdots \widehat{s_i}\cdots \widehat{s_j}\cdots s_l$. Then
\[
\widetilde{w}=wr_{\alpha+\beta}r_\alpha=w(r_\alpha r_\beta r_\alpha) r_\alpha = wr_\alpha r_\beta.
\]
Now by assumption $w\lessdot wr_\beta$, but $\widetilde{w}=wr_\beta (r_\beta r_\alpha r_\beta)=wr_\beta r_{\beta+\alpha}$ is a covering relation in Bruhat order since $l(wr_\beta)=l(\widetilde{w})-1$.
\end{proof}

\subsection{Toppings}\label{sec:toppings}
In this section we will describe a way to find all the ``minimal'' flag manifolds $H/B_H$ in which a pizza can have a Kazhdan-Lusztig atlas.
 
Assume that a pizza has a Kazhdan-Lusztig atlas in $H/B_H$, i.e. $M^\prime\subseteq H/B_H$. If a simple root $\alpha$ does not appear as a summand in any of the edge labels, then we could replace $H$ by a smaller group $H^\prime$ by removing $\alpha$ from $H$'s Dynkin diagram. Since $\alpha$ did not appear on any edge labels, $s_\alpha$ does not appear in any of the vertex labels, and the same vertex labels define a Kazhdan-Lusztig atlas in $H^\prime/B_{H^\prime}$. Therefore, to find a minimal $H$, we should look at all possible ways a simple root can appear in the edge labels of the pizza. We first look at how a simple root can label edges of individual pieces. Recall that (Lemma \ref{lem:rightmult}) the edge labels represent homology classes of the invariant $\bC\bP^1$s of the pieces. Since all our pieces are toric, we know what the relations between the classes of the edges are from the Jurkiewicz-Danilov theorem (\cite{CLS}, Theorem 12.4.4). We represent ways of a simple root appearing as edge labels of a piece by drawing a curve across the edges where it does so. Figures \ref{fig:smoothtopping} and \ref{fig:singulartopping} show all the possible ways.
\begin{figure}[H]
\centering
\includegraphics{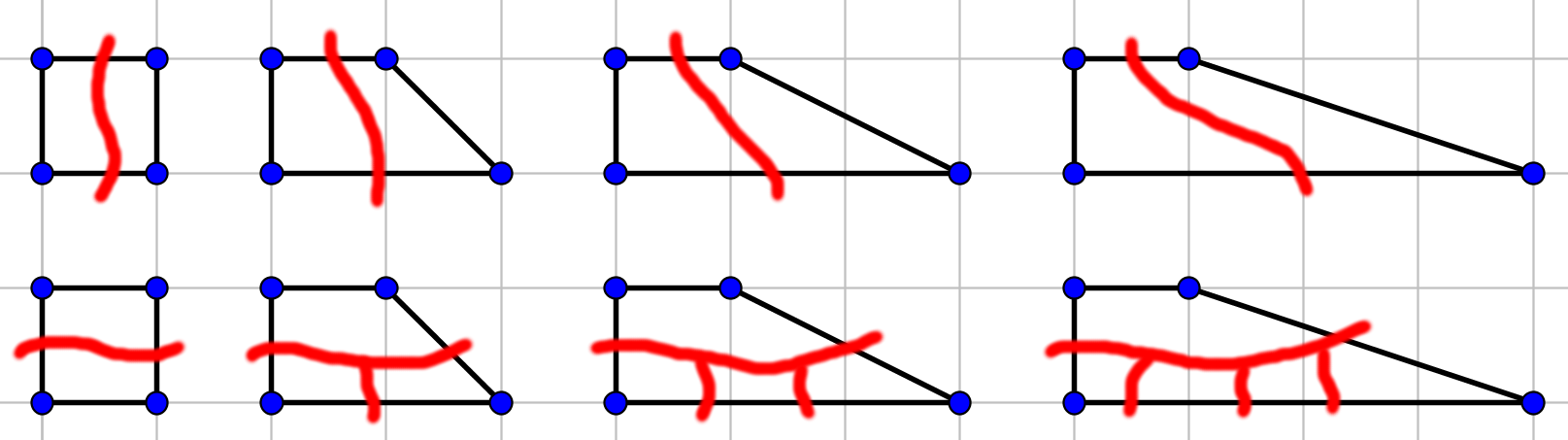}\caption{Smooth pieces\label{fig:smoothtopping}}
\end{figure}
\begin{definition}
A \textbf{topping} on a pizza piece is a generator of $H_2^{\text{effective}}$ of the pizza slice, a \textbf{compatible topping configuration} is a compatible set of toppings on the pieces of a pizza.
\end{definition}
\begin{figure}[H]
\centering
\includegraphics{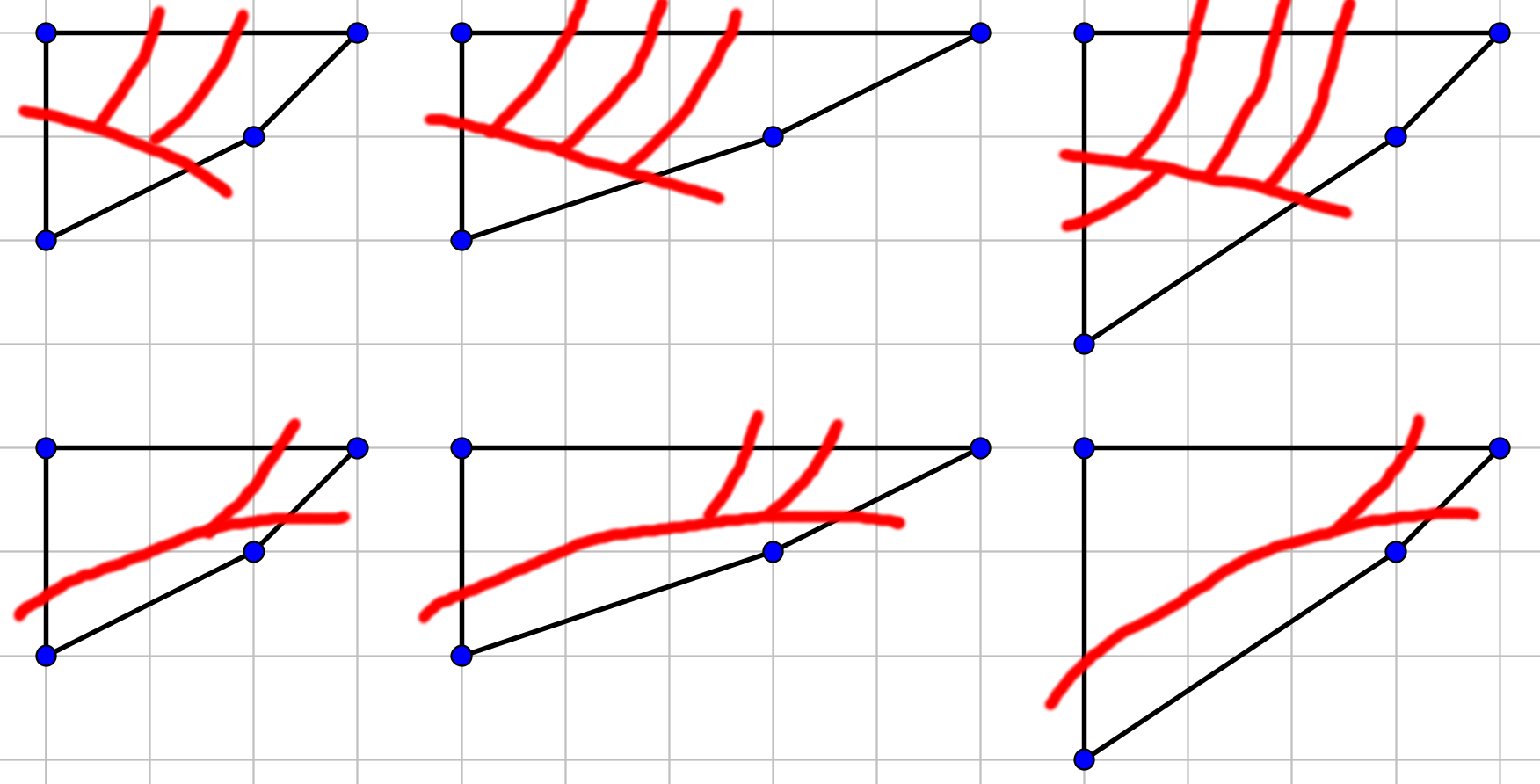}\caption{Singular pieces\label{fig:singulartopping}}
\end{figure}
A simple root of $H$ then appears as a summand on edge labels for a compatible topping configuration on the pizza, e.g.
\begin{figure}[H]
\centering
\includegraphics{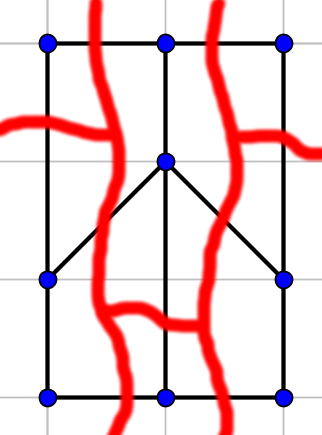}
\end{figure}
So our strategy is the following:
\begin{itemize}
\item List all possible compatible topping configurations on the pizza.
\item Choose a minimal subset of them.
\item Let $H$ be the group with precisely those simple roots.
\end{itemize}
To provide an example, we will perform this on the ``sad face'' pizza just above. The compatible topping configurations are:
\begin{figure}[H]
\centering
\includegraphics{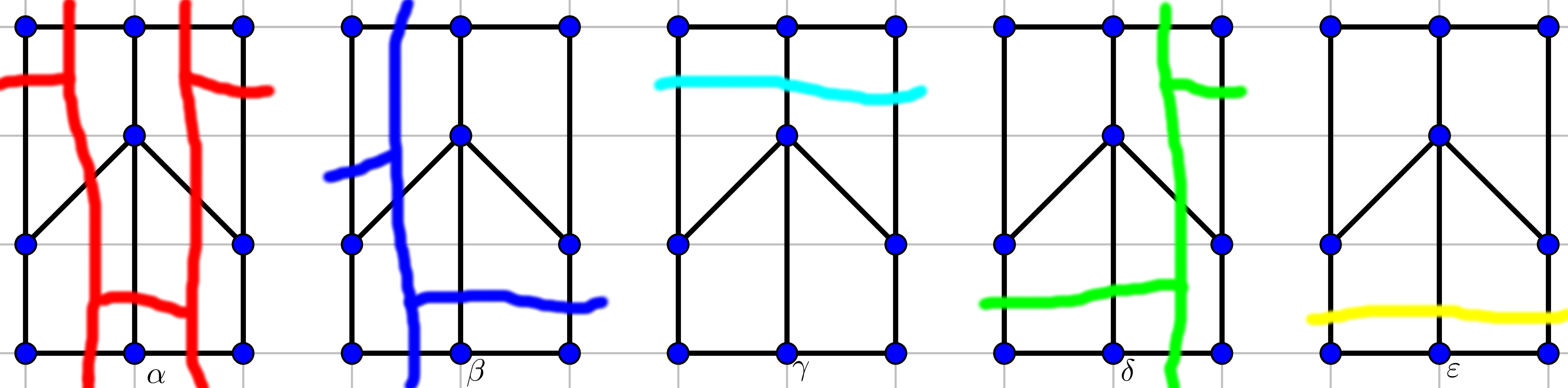}
\end{figure}
We have labeled the topping configurations by the simple roots that they represent. Note that more than one simple root may have the same topping configuration. Next, we put all labels on the pizza (note that we have to increase almost all edge lengths for this):
\begin{figure}[H]
\centering
\begin{subfigure}[b]{0.4\textwidth}
\includegraphics{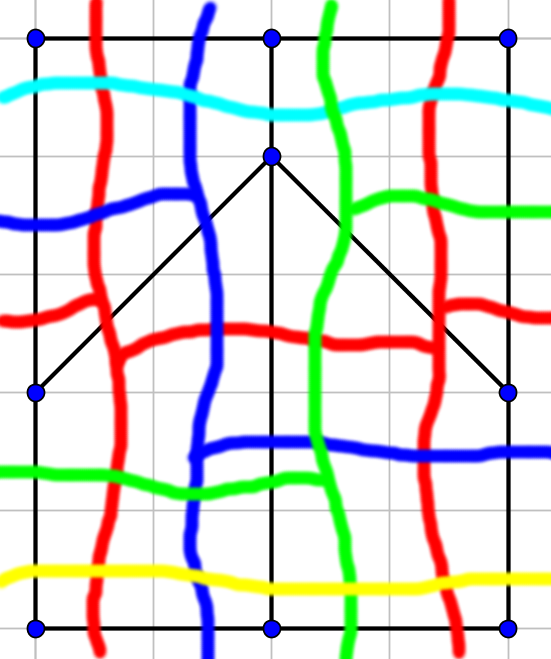}
\end{subfigure}
\begin{subfigure}[b]{0.4\textwidth}
\includegraphics{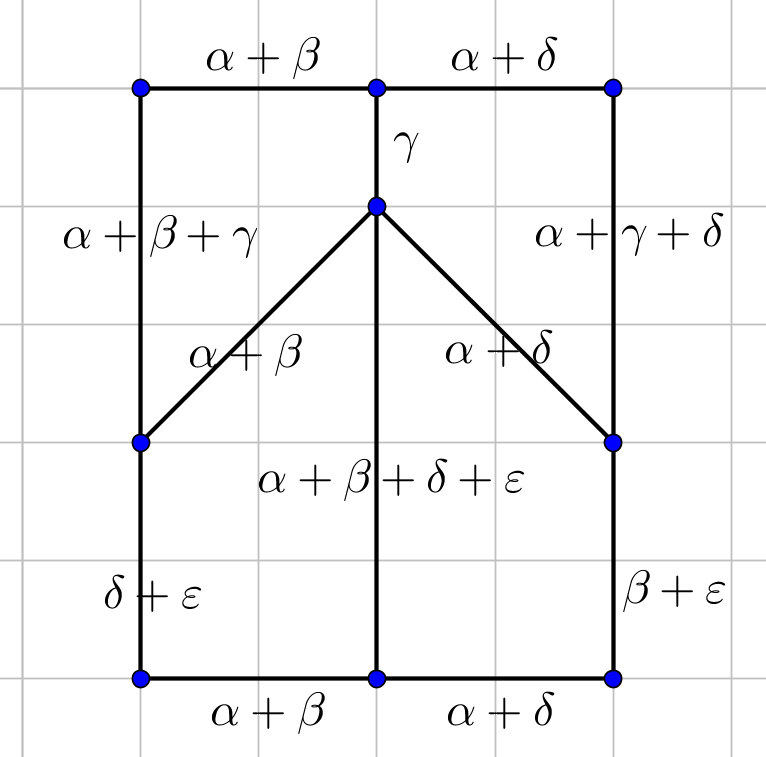}
\end{subfigure}
\end{figure}

\begin{definition}\label{def:spoke}
The \textbf{spokes} of a pizza are the edges connected to the central vertex.
\end{definition}

Now we should choose a minimal subset of these toppings. Our subset should satisfy the following:
\begin{enumerate}
\item Every edge has a topping on it.\label{nocontraction}
\item The roots labeling the edges all are real roots in $H$, since they correspond to reflections in $W_H$.\label{realroots}
\item No two spokes have the same label (so we do not contradict part $2$ of definition \ref{def:KLatlas}).
\item No three spoke labels contradict Lemma \ref{lem:diamond}.
\end{enumerate}
Such a choice is a good candidate for a minimal $H$.

\begin{enumerate}
\item To satisfy condition \ref{nocontraction}, we see that we need $\gamma$ for sure. Upon closer inspection, we may conclude that we need at least one of $\{\alpha, \varepsilon\},\{\beta,\delta\}$ to be a subset of the simple roots. 
\item To satisfy condition \ref{realroots}, we see that we can not have $\alpha, \beta, \delta, \varepsilon$ simultaneously. The reason for this is that $\alpha+\beta, \beta+\varepsilon, \varepsilon+\delta, \delta+\alpha$ are all (real) roots, since they label edges of the pizza. So $\{\alpha, \beta\}, \{\beta, \varepsilon\}, \{\varepsilon\,{\delta}\}, \{\delta, \alpha\}$ all form root systems of type $A_2$, so altogether $\{\alpha, \beta, \delta, \varepsilon\}$ must form a root system of type $\widetilde{A_3}$, in which $\alpha+\beta+\delta+\varepsilon$ is an imaginary root.
\end{enumerate}

Now we will analyze each of the cases.
\begin{itemize}
\item If $\alpha$ is not a simple root, then we must have at least $\beta,\gamma,\delta$ as simple roots. Using only these three, we see that we violate Lemma \ref{lem:diamond}. So we have to also use $\varepsilon$. So our labels must be
\begin{figure}[H]
\centering
\begin{subfigure}[b]{0.4\textwidth}
\includegraphics{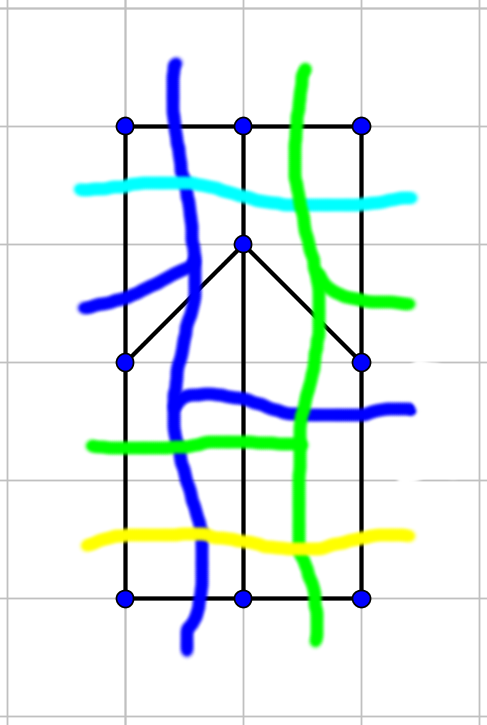}
\end{subfigure}
\begin{subfigure}[b]{0.4\textwidth}
\includegraphics{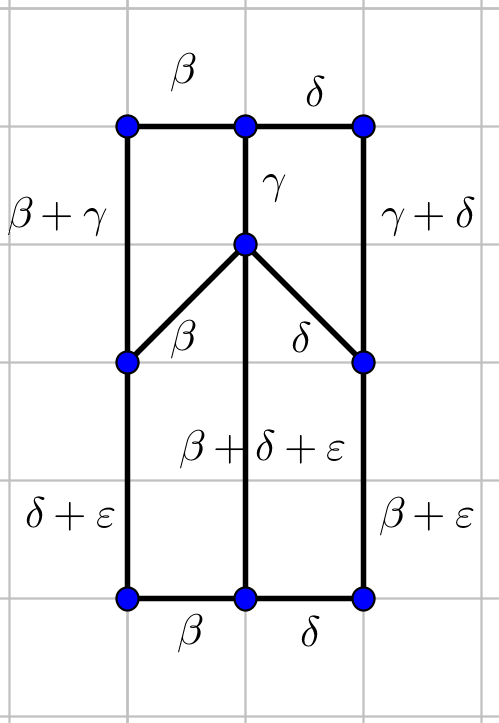}
\end{subfigure}
\end{figure}

\noindent Since $\beta+\varepsilon, \varepsilon+\delta, \delta+\gamma, \gamma+\beta$ are all roots, $\{ \beta, \gamma, \delta, \varepsilon \}$ form a root system of type $\widetilde{A_3}$. Now we should find an element $w_M(\text{center})$ that labels the central vertex of the pizza such that all the edges correspond to covering relations. If $w_M(\text{center})=s_\delta s_\beta s_\varepsilon$, then we have a Kazhdan-Lusztig atlas on this pizza.
\item If $\varepsilon$ is not a simple root, we see that we need all of $\{\alpha,\beta,\gamma,\delta\}$ to be simple roots. So our labels must be
\begin{figure}[H]
\centering
\begin{subfigure}[b]{0.4\textwidth}
\includegraphics{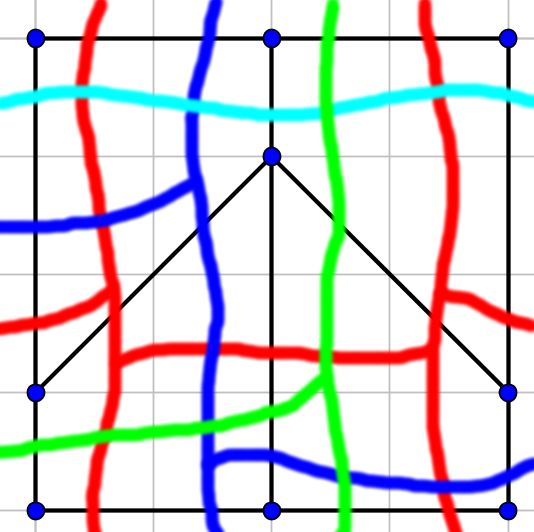}
\end{subfigure}
\begin{subfigure}[b]{0.4\textwidth}
\includegraphics{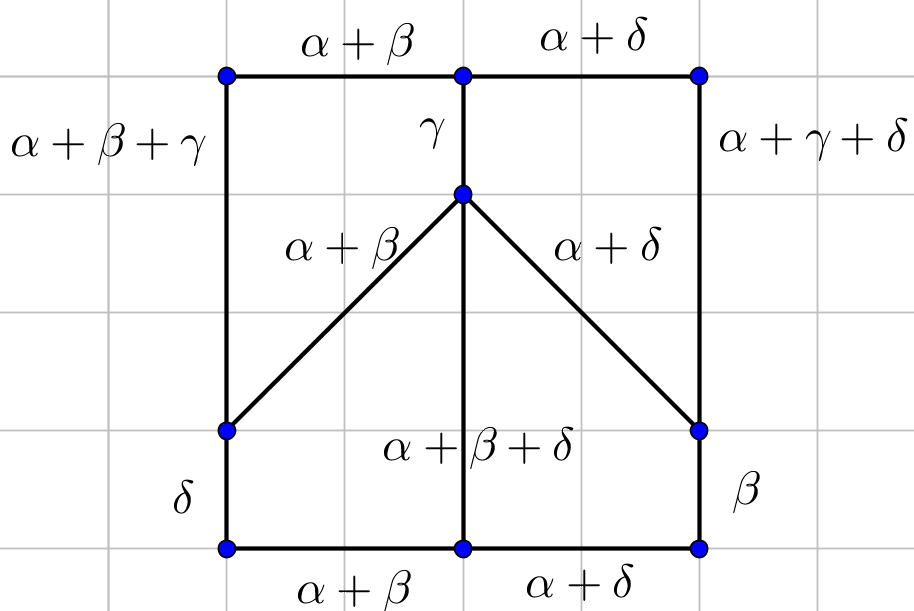}
\end{subfigure}
\end{figure}
In $H$'s Dynkin diagram, there should be edges $\{\alpha,\beta\}$, $\{\alpha,\delta\}$, and either an edge $\{\alpha,\gamma\}$, or both $\{\beta,\gamma\}$ and $\{\gamma,\delta\}$. Let us choose the $\{\alpha,\gamma\}$ edge so $H$ is of type $D_4$. Choosing $w_m(\text{center)}=s_\beta s_\delta$ yields a Kazhdan-Lusztig atlas.
\item Note that the pizza has a symmetry which exchanges $\beta$ and $\delta$, so it suffices to look at the case when $\delta$ is not a simple root. Again, we need all the remaining roots, so our labels must be
\begin{figure}[H]
\centering
\begin{subfigure}[b]{0.4\textwidth}
\includegraphics{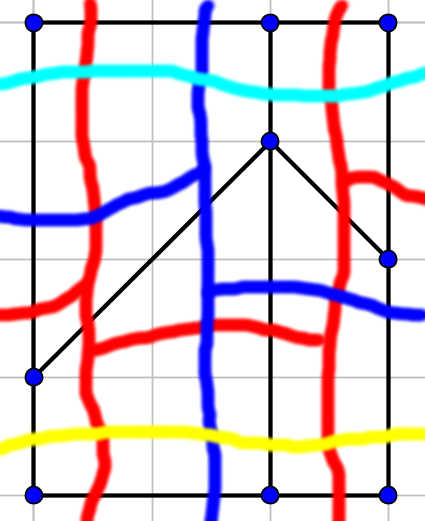}
\end{subfigure}
\begin{subfigure}[b]{0.4\textwidth}
\includegraphics{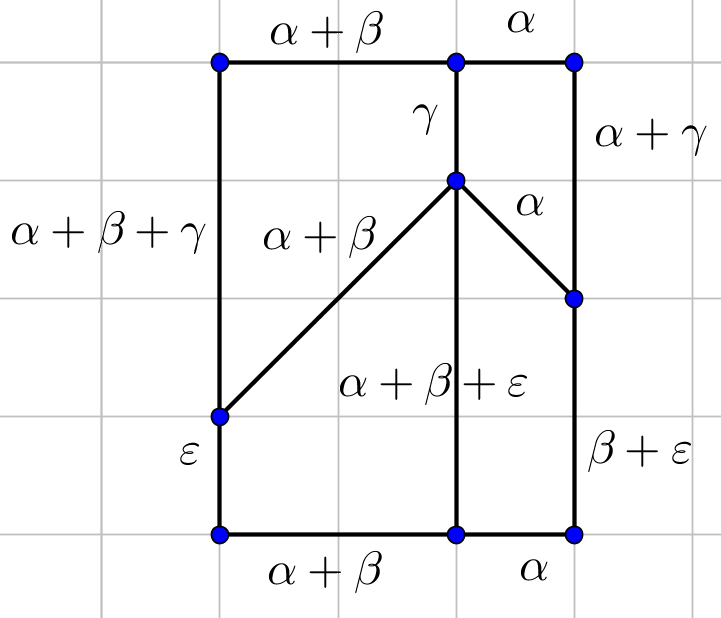}
\end{subfigure}
\end{figure}
Since $\beta+\varepsilon$, $\alpha+\beta$, $\alpha+\gamma$ must be roots, we may choose the root system to be of type $A_4$ in the following way:
\begin{figure}[H]
\centering
\includegraphics{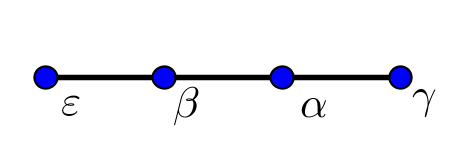}
\end{figure}
Choosing $w_m(\text{center})=s_\beta s_\varepsilon s_\beta$ yields a Kazhdan-Lusztig atlas.
\end{itemize}
So in this case, any compatible topping configuration leads to a Kazhdan-Lusztig atlas. We do not know if this is always true.

\section{Non-simply laced pizzas}\label{sec:nonsimplylacedpizzas}
The restriction to simply laced pizzas was made primarily to avoid having to deal with the $G_2^{\text{opp}}$ piece. Since it has nutritive value $\frac{0}{12}$, one fears it might appear arbitrarily many times in a pizza. In fact it can, as
\begin{proposition}\label{thm:infinitepizzafamily}
For $k\in\bN$, the sequence of pieces $[(G_2^{\text{opp}}b)^k, B_2^{\text{opp}}b, A_2^{\text{opp}}b,A_1\times A_1,(G_2^{\text{opp}})^k, B_2^{\text{opp}}, A_2^{\text{opp}},A_1\times A_1]$ is a valid pizza.
\end{proposition}
\begin{proof}
Recall that, as elements of $Br_3$, $A_2^{\text{opp}}=AB, B_2^{\text{opp}}=B^{-1}AB, G_2^{\text{opp}}=B^{-1}B^{-1}AB$ (and their backwards analogs are the same braids read backwards). For $k=0$ this is a pizza by direct checking. The general case follows from the fact that in $Br_3$,
\begin{align*}
G_2^{\text{opp}} B_2^{\text{opp}} A_2^{\text{opp}}&= B^{-1}B^{-1}ABB^{-1}ABAB \\
&=B^{-1}B^{-1}AABAB \\
&=B^{-1}B^{-1}ABABB \\
&=B^{-1}B^{-1}BABBB \\
&=B^{-1}ABBB \\
&=B^{-1}ABABB^{-1}A^{-1}BB \\
&=B_2^{\text{opp}} A_2^{\text{opp}} (G_2^{\text{opp}})^{-1},
\end{align*}
which implies
\[
G_2^{\text{opp}} B_2^{\text{opp}} A_2^{\text{opp}} G_2^{\text{opp}} = B_2^{\text{opp}} A_2^{\text{opp}}.
\]
Then our proposition follows from the fact that
\[
G_2^{\text{opp}} (A_1\times A_1) = (A_1\times A_1) G_2^{\text{opp}}b
\]
\end{proof}
However, not all of these will be labelable. Using results of Dyer (\cite{D2}) we are able to reduce the general case to a finite problem. Since only the $G_2^{\text{opp}}$ and $G_2^{\text{opp}}b$ pieces have nutritive value $\frac{0}{12}$, if we can show that a pizza can not have a Kazhdan-Lusztig atlas if it has too many of these pieces, we would be again left with a finite problem.
\begin{proposition}\label{thm:atlaspizzasfinite}
If a $G_2^{\text{opp}}$ or $G_2^{\text{opp}}b$ piece is adjacent to a $B_2^{\text{opp}},B_2^{\text{opp}}b, G_2^{\text{opp}}$ or $G_2^{\text{opp}}b$ piece in a pizza, then the pizza can not have an atlas. Note that this implies that no two $\frac{0}{12}$ nutritional pieces appear consecutively. Also, if a pizza has an atlas, then the only piece sequence in which two $B_2^{\text{opp}}$ or $B_2^{\text{opp}}b$ pieces can be adjacent to each other is $B_2^{\text{opp}},B_2^{\text{opp}}b$.
\end{proposition}
\begin{proof}
We will check the $G_2^{\text{opp}}b, G_2^{\text{opp}}$ case; the other cases are very similar to this one. The sequence of slices looks like (the central vertex is highlighted in red)
\begin{figure}[H]
\centering
\includegraphics{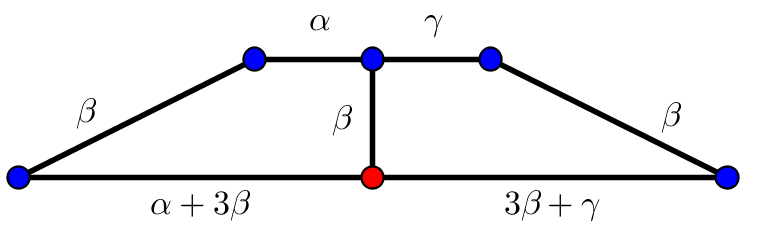}
\end{figure}
Note that we do not know the heights of the roots $\alpha,\beta,\gamma$, but we know from Proposition \ref{thm:richardsonquads} and the discussion afterwards that both $\alpha,\beta$ and $\beta,\gamma$ must form root systems of type $G_2$, with $\beta$ being the short root in both of them.

If $w\in W_H$ is an element covered by $wr_\beta, wr_{\alpha+3\beta}, wr_{3\beta+\gamma}$, then $w$ must move all the following roots to negatives: $\{\alpha, \alpha+\beta, 2\alpha+3\beta, \alpha+2\beta,\gamma, \beta+\gamma, 3\beta+2\gamma, 2\beta+\gamma \}$, since by Theorem 1.4. of \cite{D2}, it suffices to check this in the reflection subgroups $W_{\alpha,\beta}=\langle r_\alpha, r_\beta \rangle, W_{\beta, \gamma}=\langle r_\beta,r_\gamma \rangle$ (both isomorphic to $W_{G_2}$). Now consider the reflection subgroup $W_{\alpha,\beta,\gamma}=\langle r_\alpha, r_\beta ,r_\gamma\rangle$. Clearly any root that is a convex combination of the roots above will be moved to negative roots, in particular, any root of the form $c_1\alpha+c_2\beta+c_3\gamma$ as long as $c_2\leq 2c_1+2c_3$. There are infinitely many roots of this form, so such $w$ would need to have infinite length, which is a contradiction.
\end{proof}
\begin{lemma}\label{thm:reflsubgrplemma}
Assume that $\alpha,\beta$ are simple roots in a root system of type $A_2\; (\text{resp. }B_2, G_2)$ with $\beta$ being the short root (if there are two root lengths). If $w\in W_{\alpha,\beta}$ such that $w\lessdot ws_\beta, w\lessdot ws_{\alpha+\beta}\; (\text{resp. }ws_{\alpha+2\beta}$, or, in the $G_2$ case, $ws_{\alpha+3\beta})$, then $w\cdot \gamma$ is negative for every element of the following set of positive roots: $\{\alpha\}\; (\text{resp. }\{\alpha,\alpha+\beta\}$, or, in the $G_2$ case, $\{\alpha,\alpha+\beta,2\alpha+3\beta,\alpha+2\beta\})$ to negative roots, or, equivalently $ws_\alpha < w\; (\text{resp. }ws_\alpha, wr_{\alpha+\beta} < w$, or, in the $G_2$ case, $ws_\alpha, wr_{\alpha+\beta}, wr_{2\alpha+3\beta}, wr_{\alpha+2\beta}<w)$.
\end{lemma}
\begin{proof}
The only element $w$ that satisfies the covering relations is $s_\alpha\; (\text{resp. }s_\beta s_\alpha,\text{ or, in the $G_2$ case, } s_\beta s_\alpha s_\beta s_\alpha)$ which also moves the above-mentioned roots to negatives.
\end{proof}

Proposition \ref{thm:atlaspizzasfinite} reduces the general case to a finite problem. The following is the best we can say at this moment:
\begin{theorem}\label{thm:atlasbound}
The number of pizzas with Kazhdan-Lusztig atlases is at most $7543$, each having at most $12$ pieces.
\end{theorem}
\begin{proof}
This is a brute-force check by Sage \cite{sage}, using Proposition \ref{thm:atlaspizzasfinite}.
\end{proof}

\section{Some Kazhdan-Lusztig atlases for simply laced pizzas}\label{sec:listofatlases}
\begin{definition}
We define the \textbf{height} of a Kazhdan-Lusztig atlas to be the length of the $W_H$-element at the central vertex of the pizza.
\end{definition}
We will only describe one of the (possibly many) minimal height atlases for each pizza. All of these atlases have been obtained following the algorithm of section \ref{sec:toppings}. The Dynkin diagram of the group $H$ is displayed near the pizza. The search for the $W_H$ element at the center of each pizza was done by Sage \cite{sage}. Regrettably, we were unable to find an atlas (even a non-simply-laced one) for the pizza $[A_2^{opp}, A_2^{oppb}, A_2^{opp}, G_2b]$, we suspect that it does not have an atlas.
\begin{enumerate}

\subsection*{Height $0$ (Bruhat atlases)}
\begin{minipage}{\textwidth}\item The pizza $[A_2, A_2, A_2]$:
\begin{figure}[H]
\centering
\begin{subfigure}{0.79\textwidth}
\includegraphics{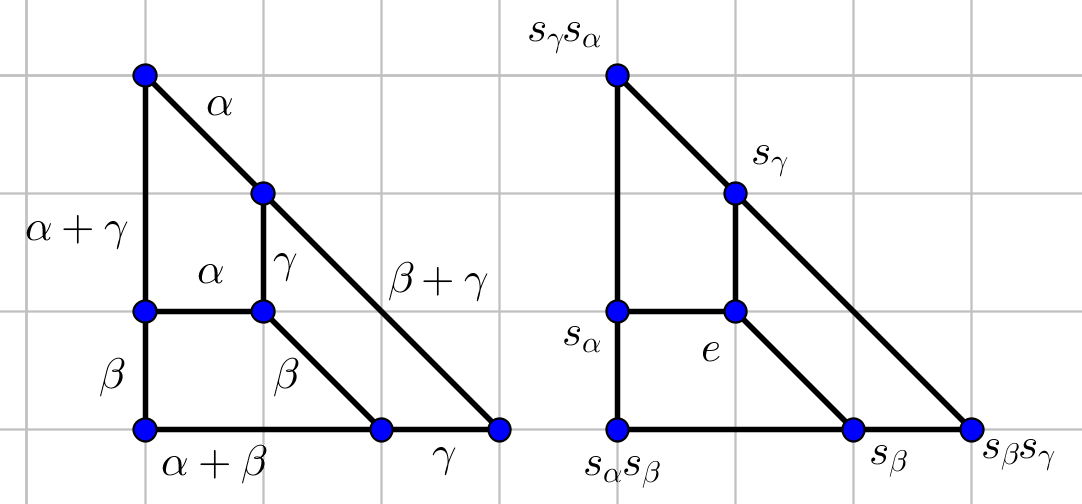}
\end{subfigure}
\begin{subfigure}{0.2\textwidth}
\includegraphics{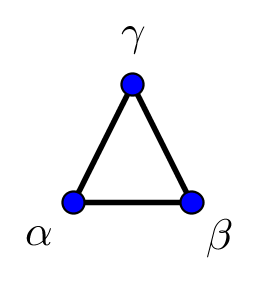}
\end{subfigure}
\end{figure}
\end{minipage}

\begin{minipage}{\textwidth}\item The pizza $[A_1\times A_1, A_1\times A_1, A_1\times A_1, A_1\times A_1]$:
\begin{figure}[H]
\begin{subfigure}{0.79\textwidth}
\includegraphics{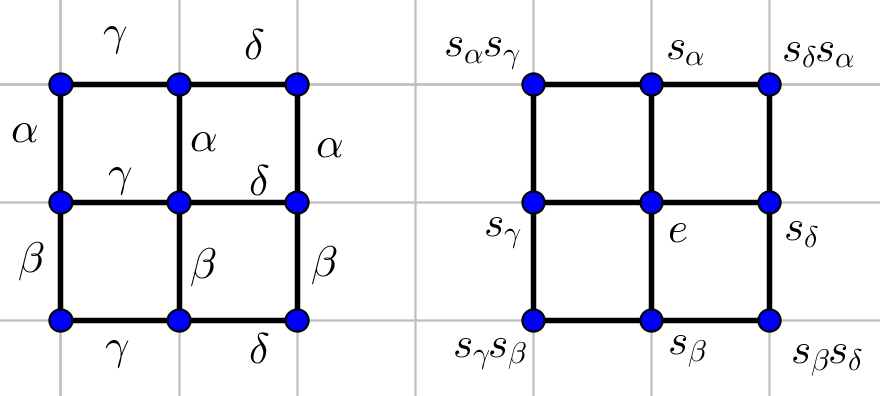}
\end{subfigure}
\begin{subfigure}{0.2\textwidth}
\includegraphics{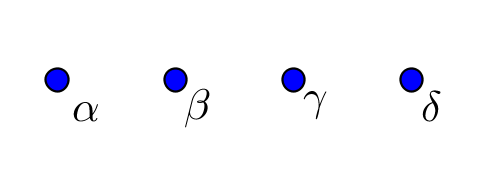}
\end{subfigure}
\end{figure}
\end{minipage}

\subsection*{Height $1$}
\begin{minipage}{\textwidth}\item The pizza $[A_1\times A_1, A_1\times A_1, A_2, A_2^{opp}b]$:
\begin{figure}[H]
\centering
\begin{subfigure}{0.79\textwidth}
\includegraphics{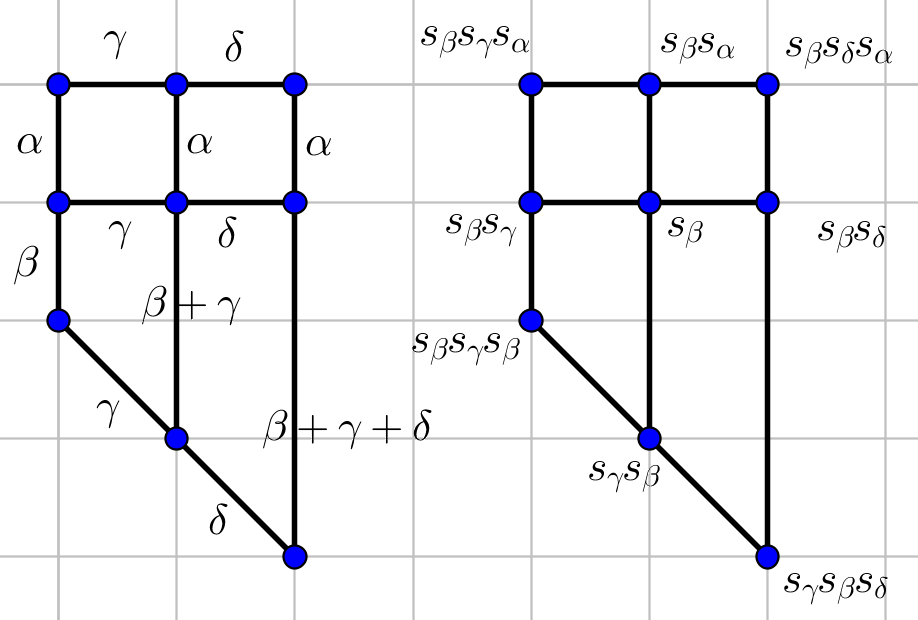}
\end{subfigure}
\begin{subfigure}{0.2\textwidth}
\includegraphics{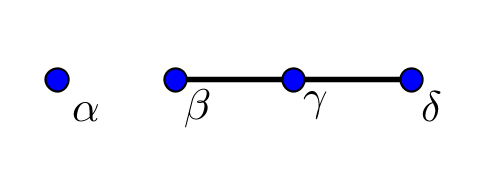}
\end{subfigure}
\end{figure}
\end{minipage}

\begin{minipage}{\textwidth}\item The pizza $[A_1\times A_1,A_1\times A_1,A_2b,A_2^{opp}]$:
\begin{figure}[H]
\centering
\begin{subfigure}{0.79\textwidth}
\includegraphics{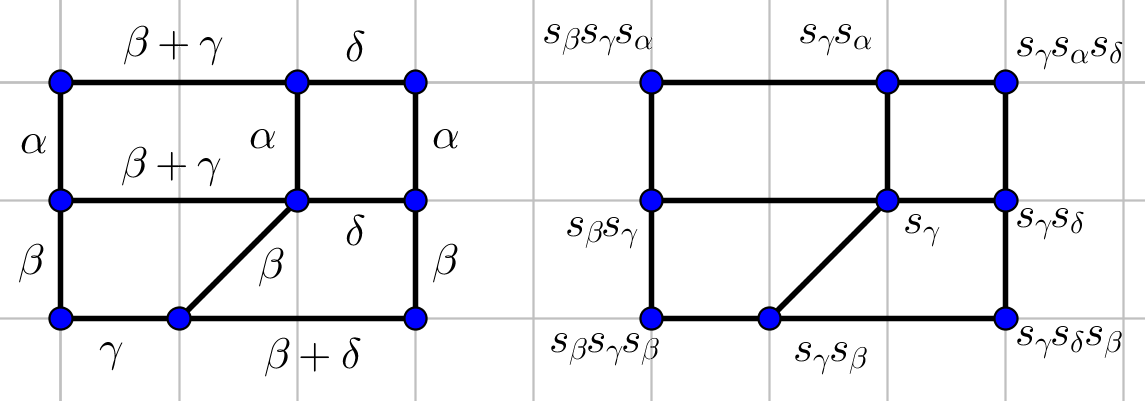}
\end{subfigure}
\begin{subfigure}{0.2\textwidth}
\includegraphics{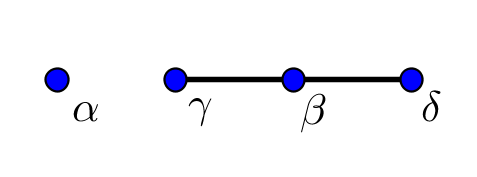}
\end{subfigure}
\end{figure}
\end{minipage}

\begin{minipage}{\textwidth}\item The pizza $[A_1\times A_1,A_2,A_1\times A_1,A_2^{opp}]$:
\begin{figure}[H]
\centering
\begin{subfigure}{0.79\textwidth}
\includegraphics{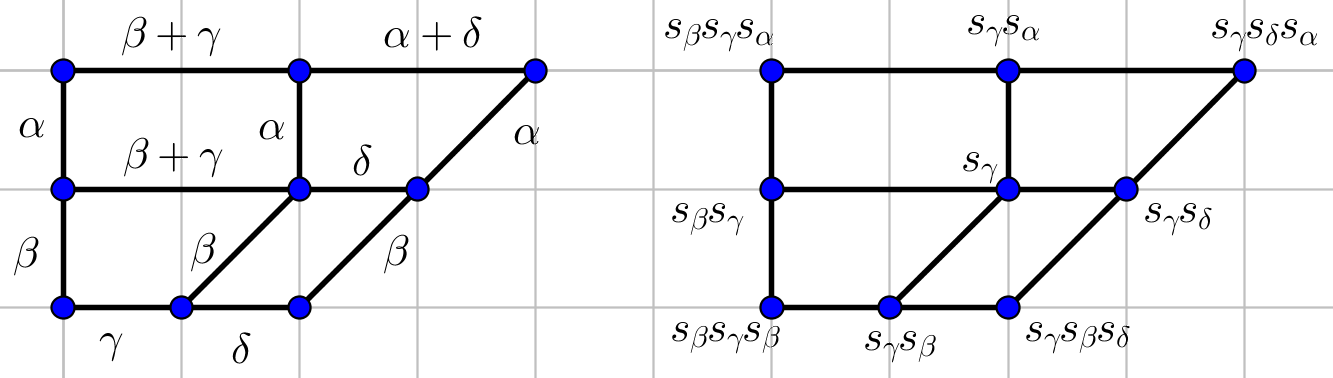}
\end{subfigure}
\begin{subfigure}{0.2\textwidth}
\includegraphics{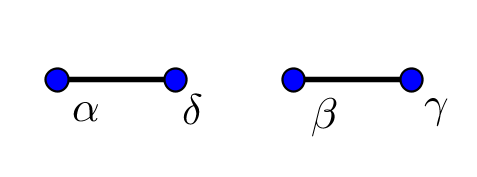}
\end{subfigure}
\end{figure}
\end{minipage}

\begin{minipage}{\textwidth}\item The pizza $[A_2, A_2,A_2^{opp}b, A_2^{opp}b]$:
\begin{figure}[H]
\centering
\begin{subfigure}{0.79\textwidth}
\includegraphics{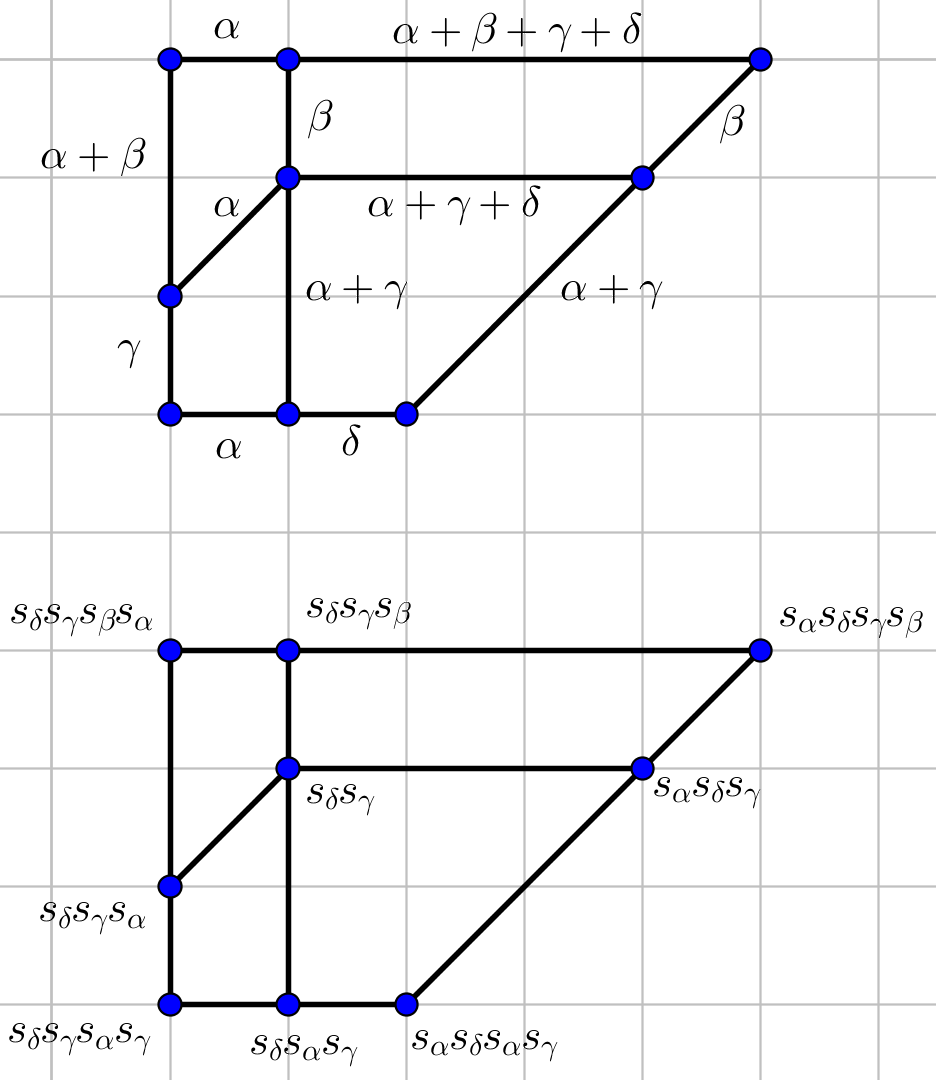}
\end{subfigure}
\begin{subfigure}{0.2\textwidth}
\includegraphics{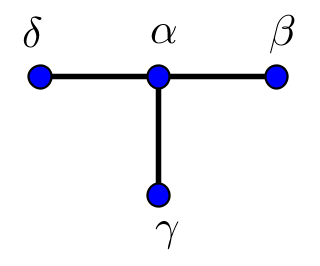}
\end{subfigure}
\end{figure}
\end{minipage}

\subsection*{Height $2$}

\begin{minipage}{\textwidth}\item The pizza $[A_2b,A_2^{opp}, A_2^{opp}b, A_2]$:
\begin{figure}[H]
\centering
\begin{subfigure}{0.79\textwidth}
\includegraphics{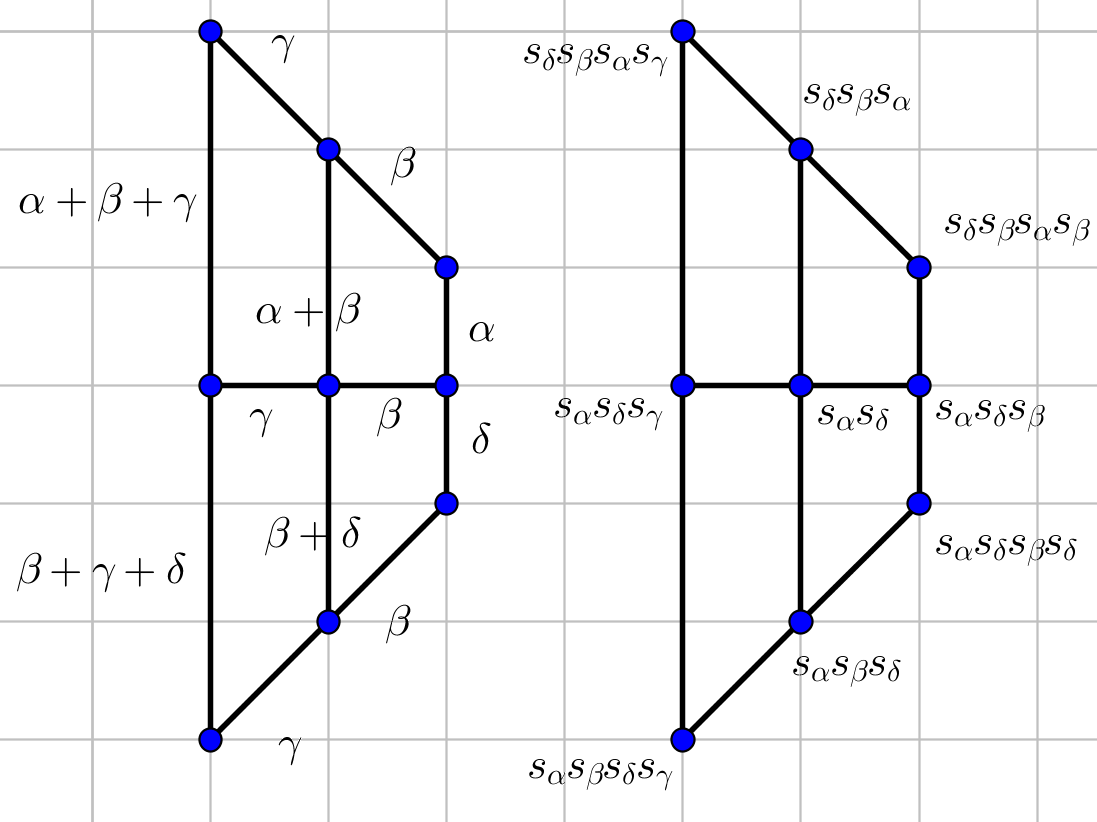}
\end{subfigure}
\begin{subfigure}{0.2\textwidth}
\includegraphics{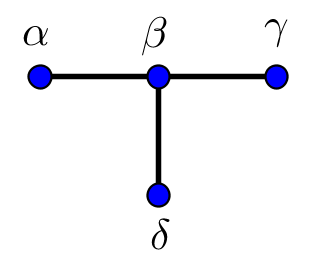}
\end{subfigure}
\end{figure}
\end{minipage}

\begin{minipage}{\textwidth}\item The pizza $[A_2,A_2^{opp}, A_2b, A_2^{opp}b]$:
\begin{figure}[H]
\centering
\begin{subfigure}{0.69\textwidth}
\includegraphics{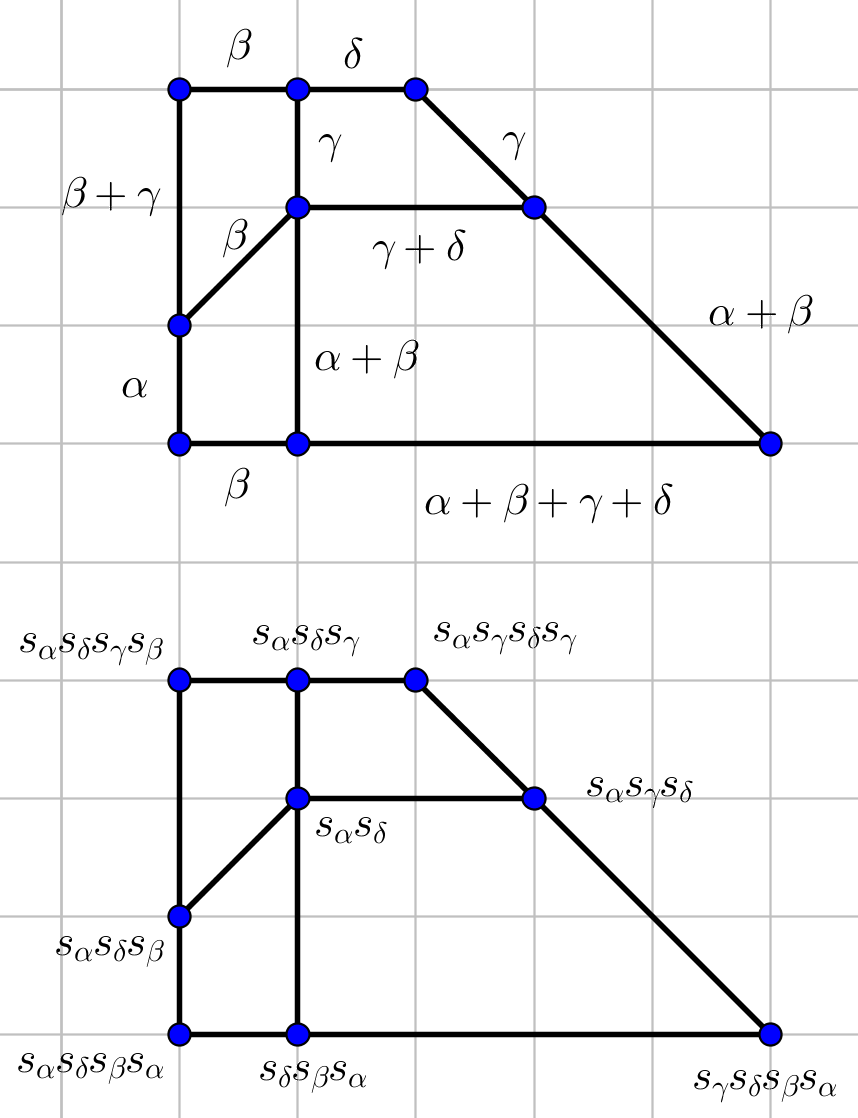}
\end{subfigure}
\begin{subfigure}{0.3\textwidth}
\includegraphics{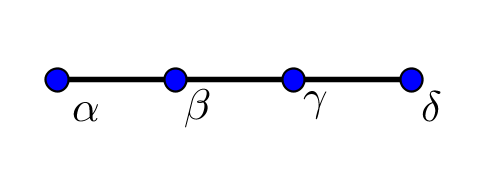}
\end{subfigure}
\end{figure}
\end{minipage}

\begin{minipage}{\textwidth}\item The pizza $[A_2^{opp}b, A_2, A_2^{opp}b, A_2]$:
\begin{figure}[H]
\centering
\begin{subfigure}{0.69\textwidth}
\includegraphics{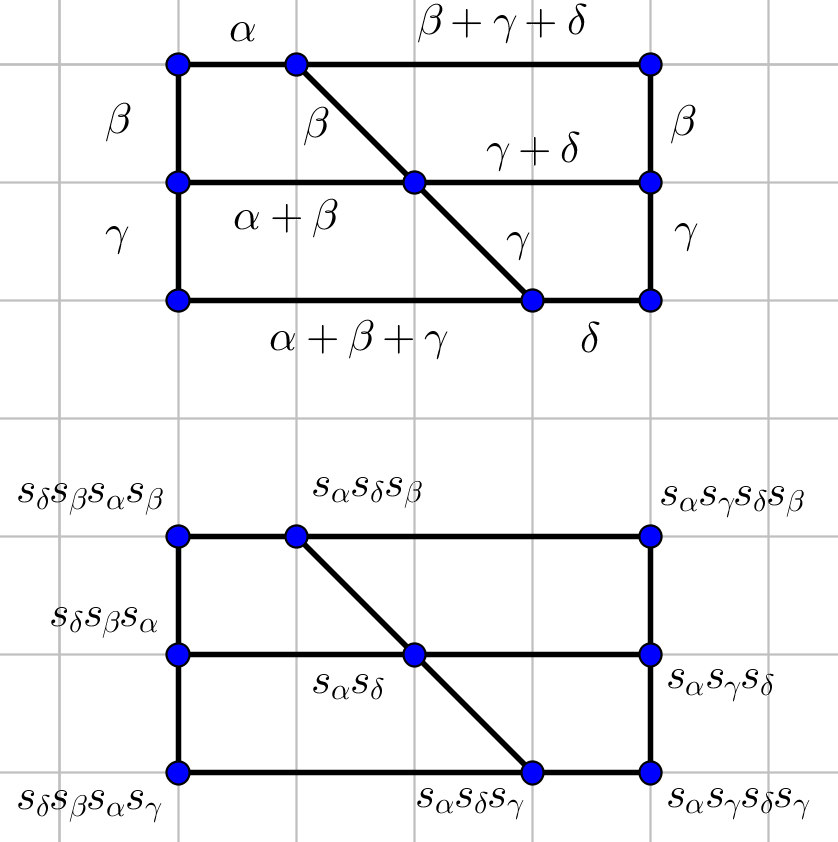}
\end{subfigure}
\begin{subfigure}{0.3\textwidth}
\includegraphics{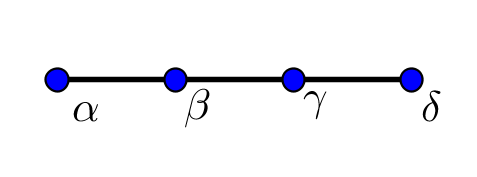}
\end{subfigure}
\end{figure}
\end{minipage}

\begin{minipage}{\textwidth}\item The pizza $[A_2^{opp}, A_2b, A_2, A_2^{opp}]$:
\begin{figure}[H]
\centering
\begin{subfigure}{0.69\textwidth}
\includegraphics{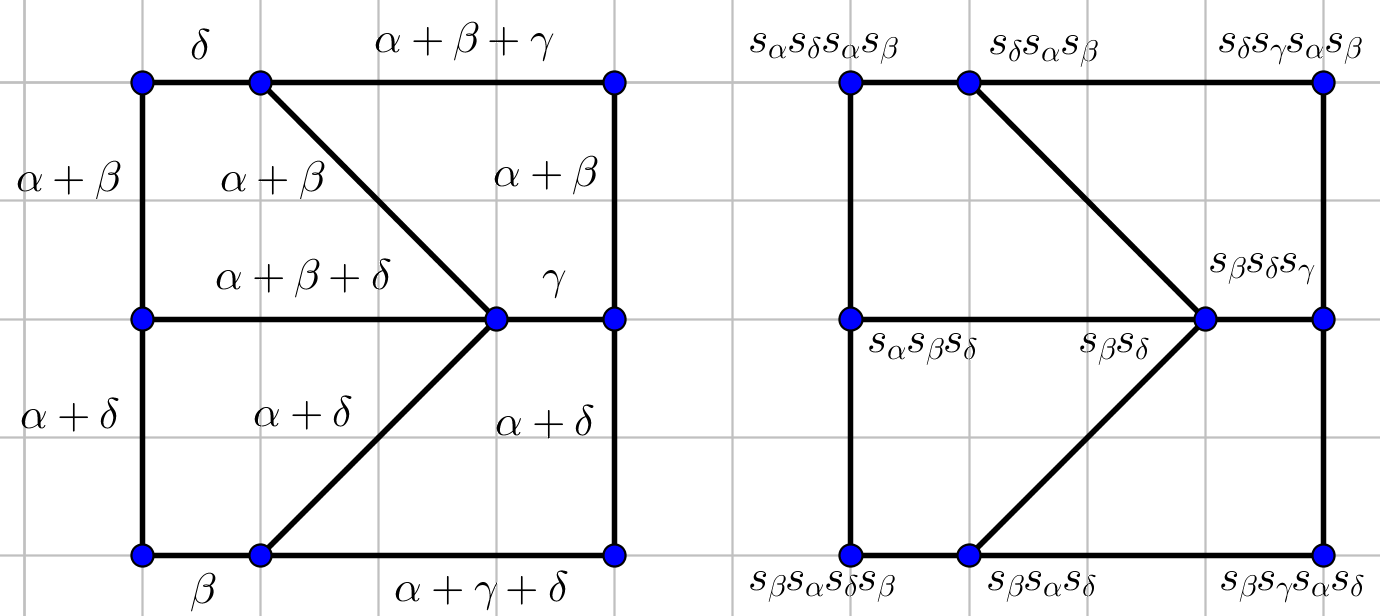}
\end{subfigure}
\begin{subfigure}{0.3\textwidth}
\includegraphics{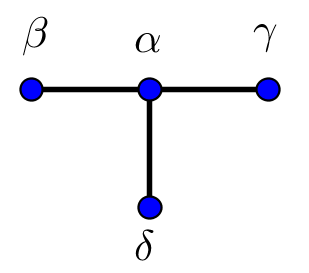}
\end{subfigure}
\end{figure}
\end{minipage}

\subsection*{Height $3$}

\begin{minipage}{\textwidth}\item The pizza $[A_1\times A_1,A_2^{opp}, A_2^{opp}b,B_2]$:
\begin{figure}[H]
\centering
\begin{subfigure}{0.79\textwidth}
\includegraphics{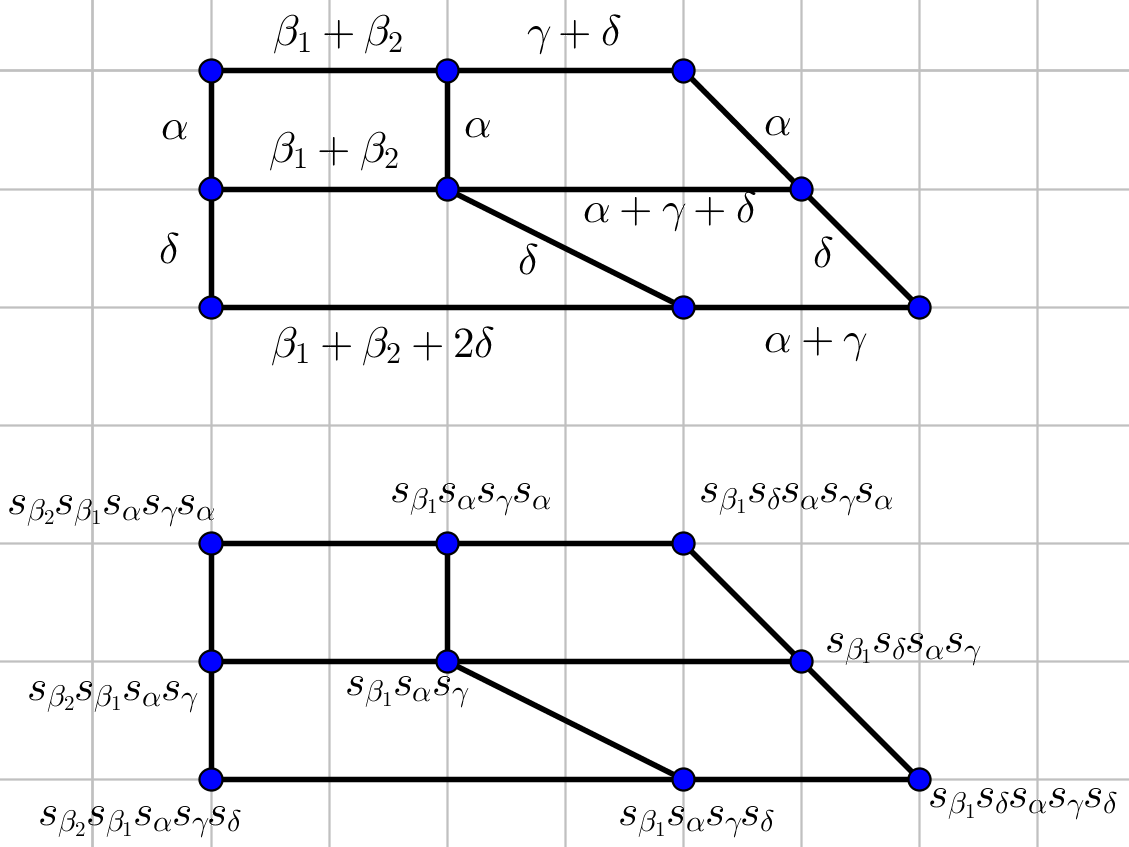}
\end{subfigure}
\begin{subfigure}{0.2\textwidth}
\includegraphics{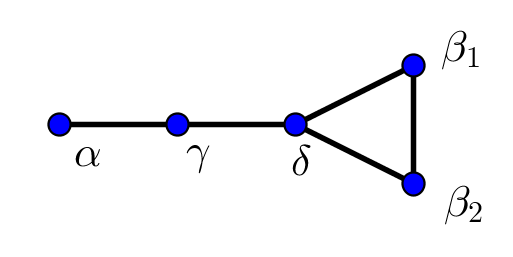}
\end{subfigure}
\end{figure}
\end{minipage}

\begin{minipage}{\textwidth}\item The pizza $[A_1\times A_1,A_2^{opp}, B_2b,A_2^{opp}]$:
\begin{figure}[H]
\centering
\begin{subfigure}{0.79\textwidth}
\includegraphics{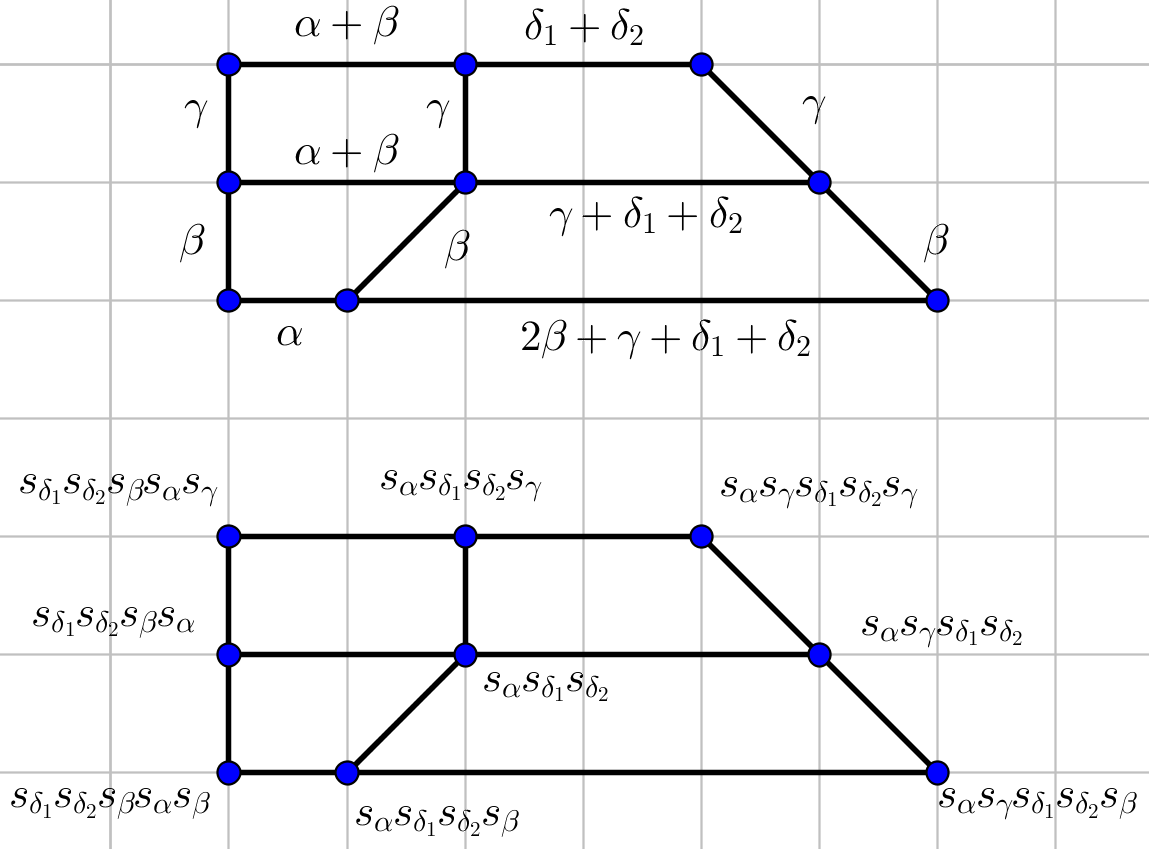}
\end{subfigure}
\begin{subfigure}{0.2
\textwidth}
\includegraphics{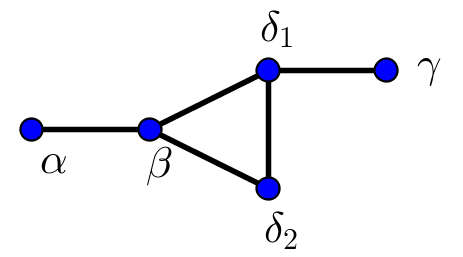}
\end{subfigure}
\end{figure}
\end{minipage}

\begin{minipage}{\textwidth}\item The pizza $[A_2^{opp}b, A_2^{opp},B_2b, A_1\times A_1]$:
\begin{figure}[H]
\centering
\begin{subfigure}{0.79\textwidth}
\includegraphics{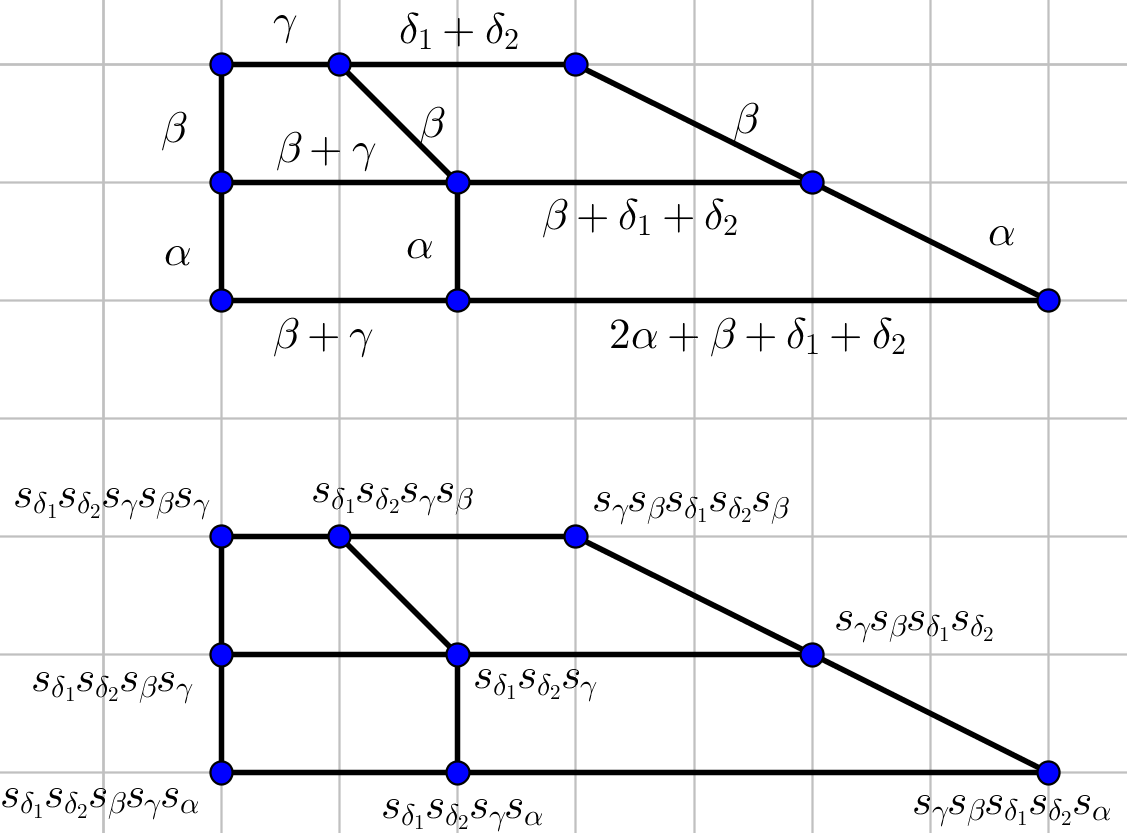}
\end{subfigure}
\begin{subfigure}{0.2\textwidth}
\includegraphics{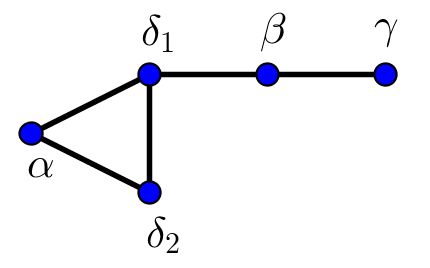}
\end{subfigure}
\end{figure}
\end{minipage}


\begin{minipage}{\textwidth}\item The pizza $[A_1\times A_1, A_1\times A_1, A_2^{opp}, A_2^{opp}, A_2^{opp}]$:
\begin{figure}[H]
\centering
\begin{subfigure}{0.69\textwidth}
\includegraphics{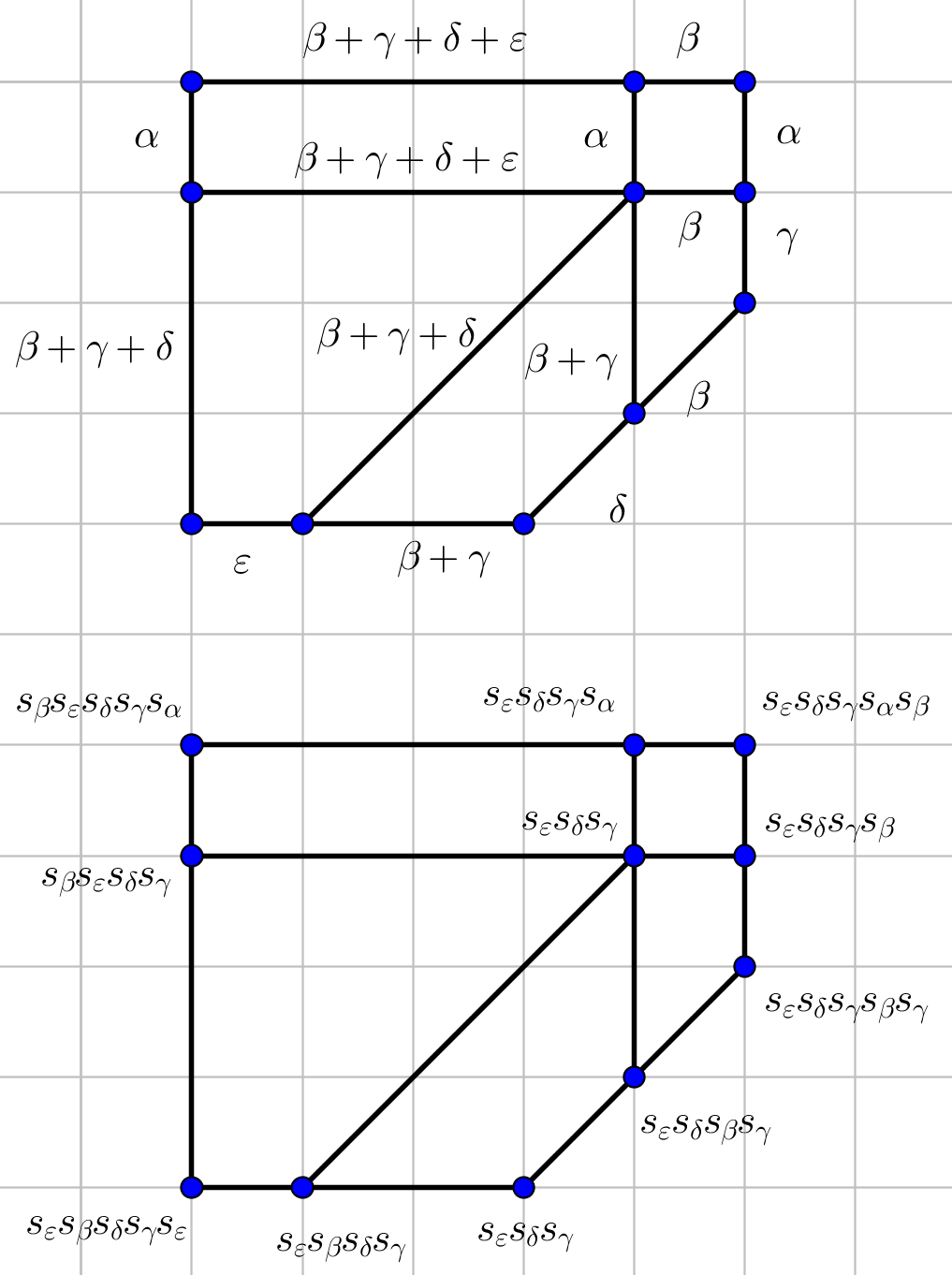}
\end{subfigure}
\begin{subfigure}{0.3\textwidth}
\includegraphics{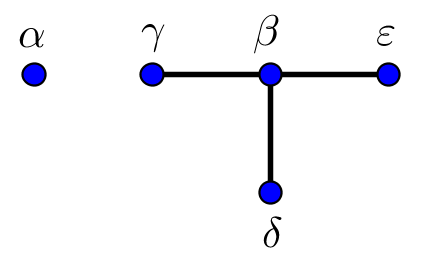}
\end{subfigure}
\end{figure}
\end{minipage}

\begin{minipage}{\textwidth}\item The pizza $[A_1\times A_1, A_2^{opp},A_1\times A_1, A_2^{opp}b, A_2^{opp}b]$:
\begin{figure}[H]
\centering
\begin{subfigure}{0.6\textwidth}
\includegraphics{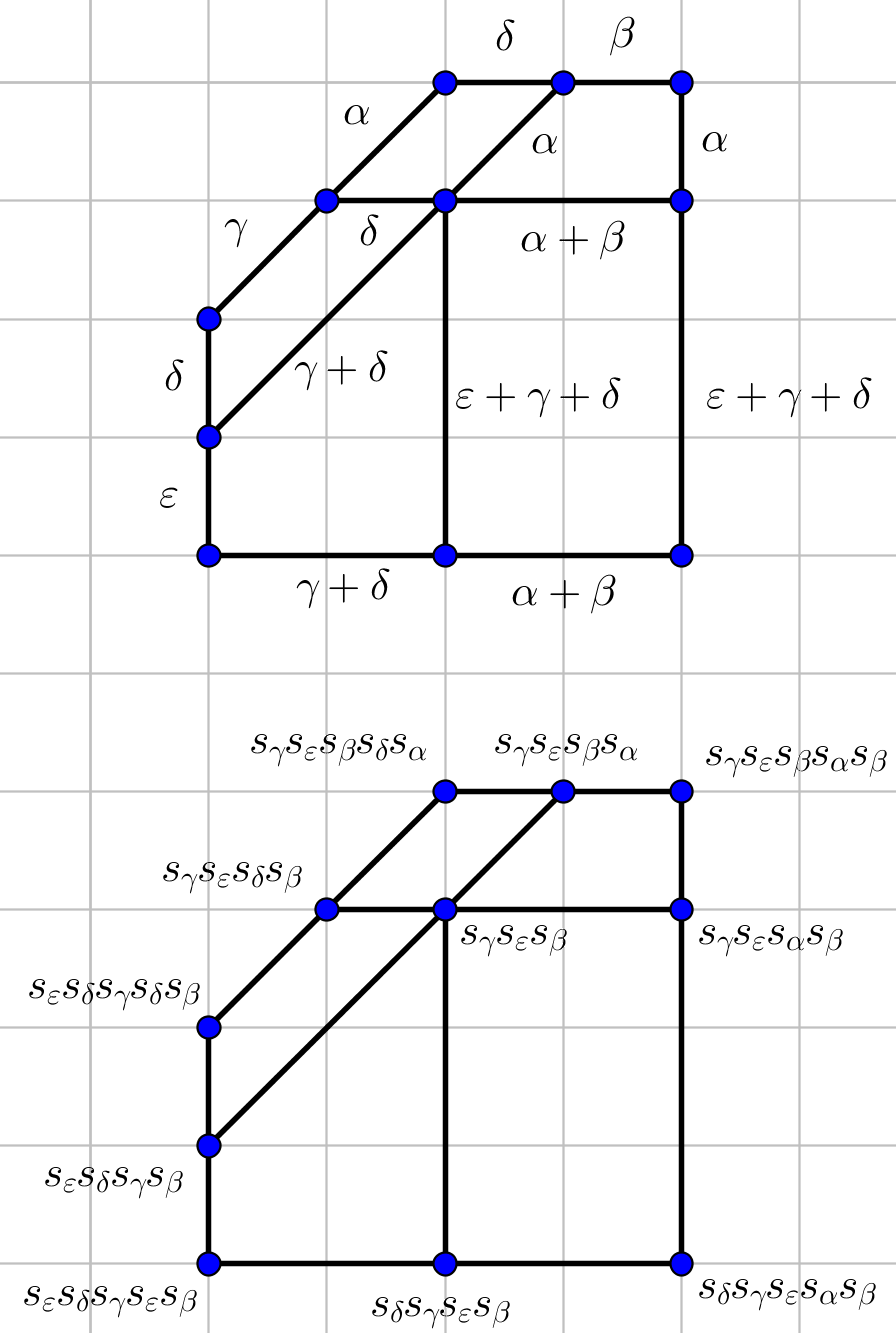}
\end{subfigure}
\begin{subfigure}{0.39\textwidth}
\includegraphics{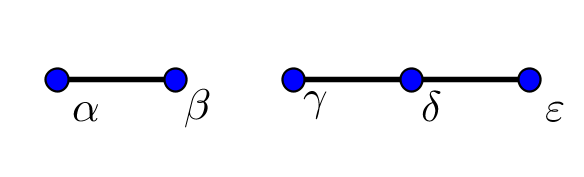}
\end{subfigure}
\end{figure}
\end{minipage}

\begin{minipage}{\textwidth}\item The pizza $[ A_2^{\text{opp}}, A_2^{\text{opp}}, A_2^{\text{opp}}, A_2^{\text{opp}}b, A_2^{\text{opp}}b, A_2^{\text{opp}}b ]$
\begin{figure}[H]
\centering
\begin{subfigure}{0.6\textwidth}
\includegraphics{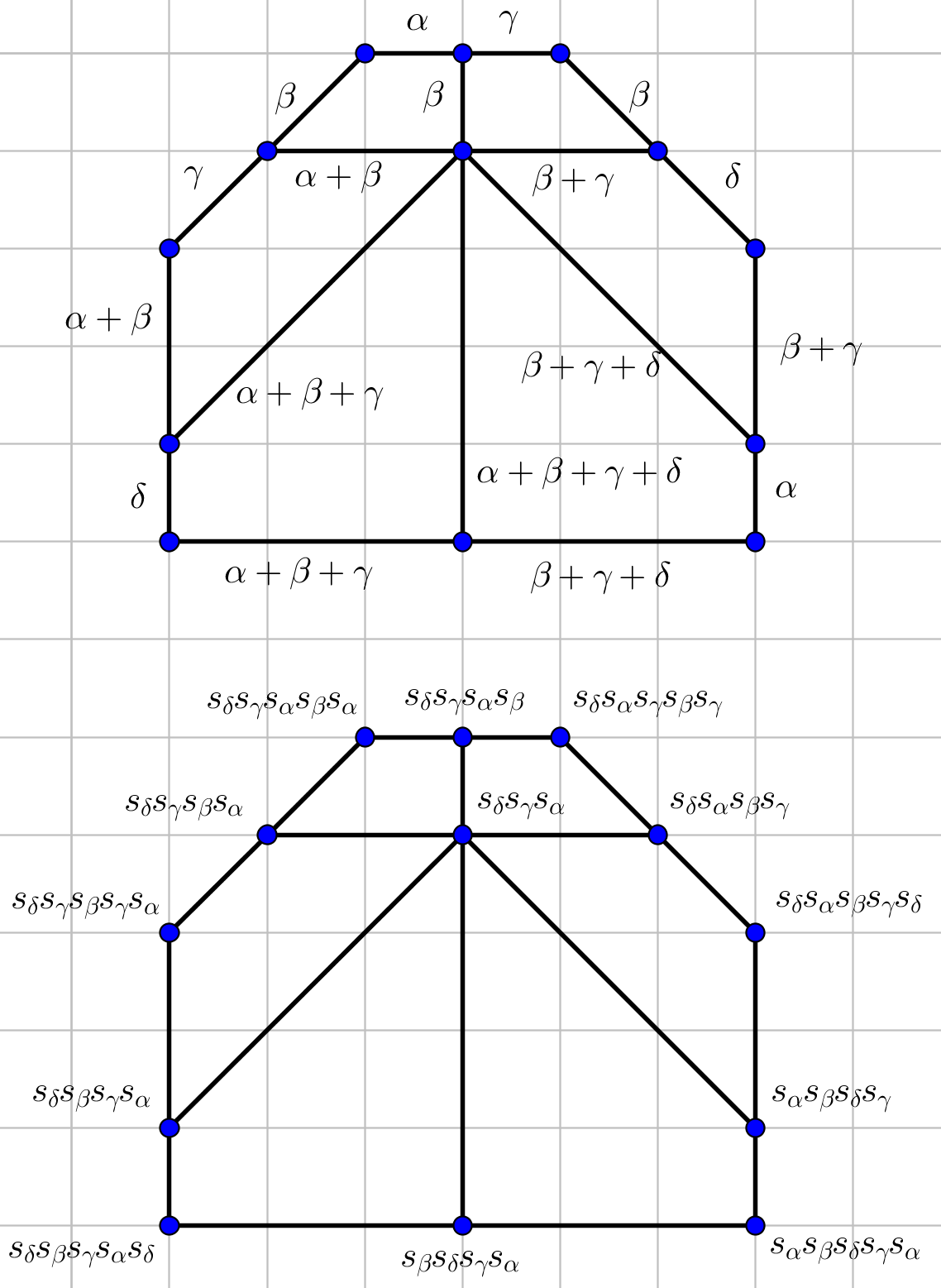}
\end{subfigure}
\begin{subfigure}{0.39\textwidth}
\includegraphics{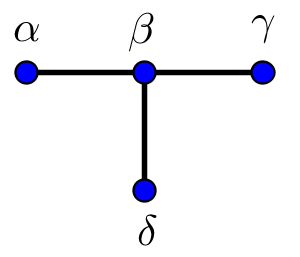}
\end{subfigure}
\end{figure}
\end{minipage}

\subsection*{Height $5$}

\begin{minipage}{\textwidth}\item The pizza $[A_2, A_2^{\text{opp}}b, A_2^{\text{opp}}b, A_2^{\text{opp}}b, A_2^{\text{opp}}b]$:
\begin{figure}[H]
\centering
\begin{subfigure}{0.6\textwidth}
\includegraphics{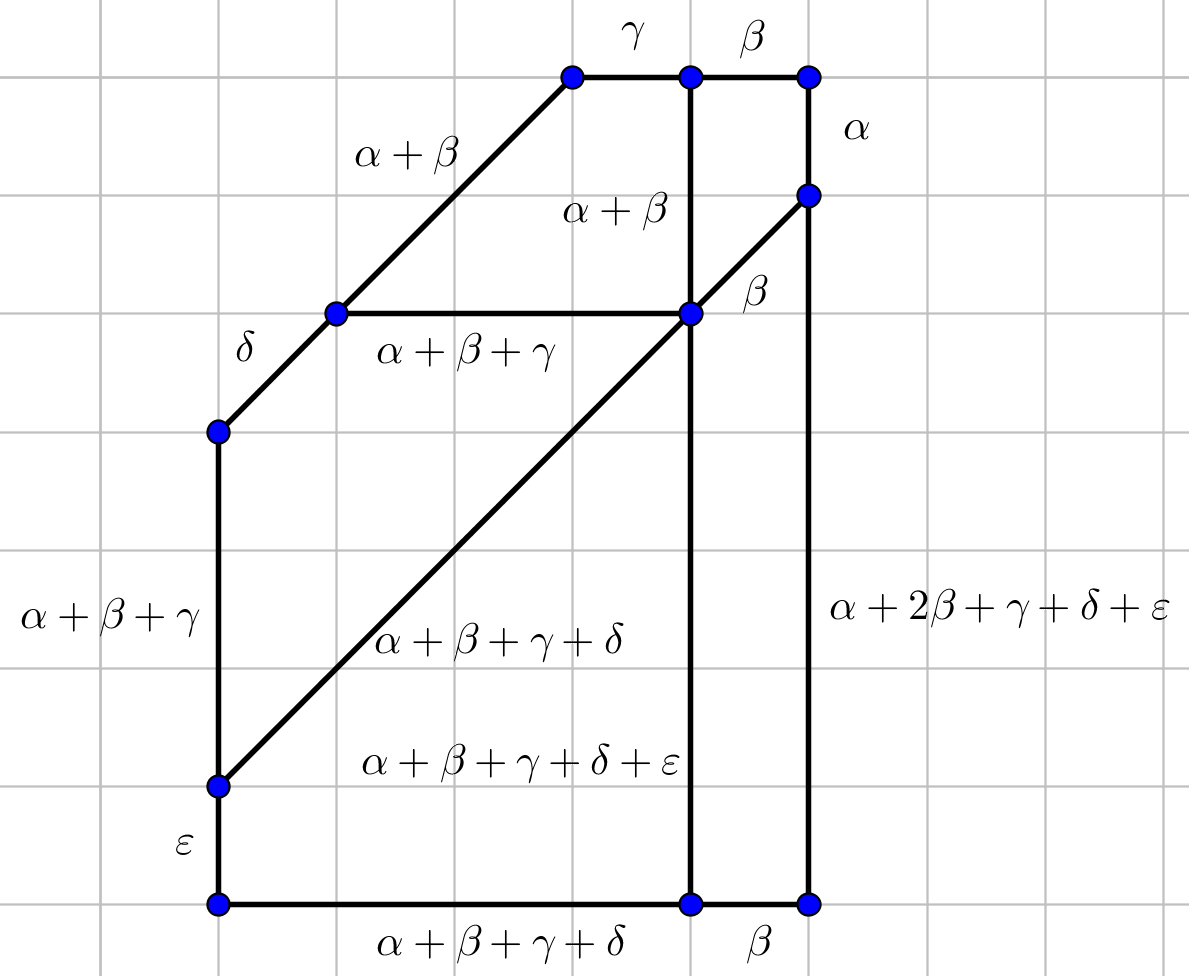}
\end{subfigure}
\begin{subfigure}{0.39\textwidth}
\includegraphics{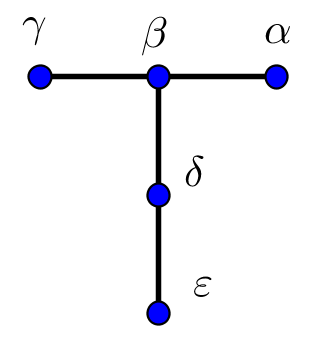}
\end{subfigure}
\end{figure}
\begin{figure}[H]
\centering
\includegraphics{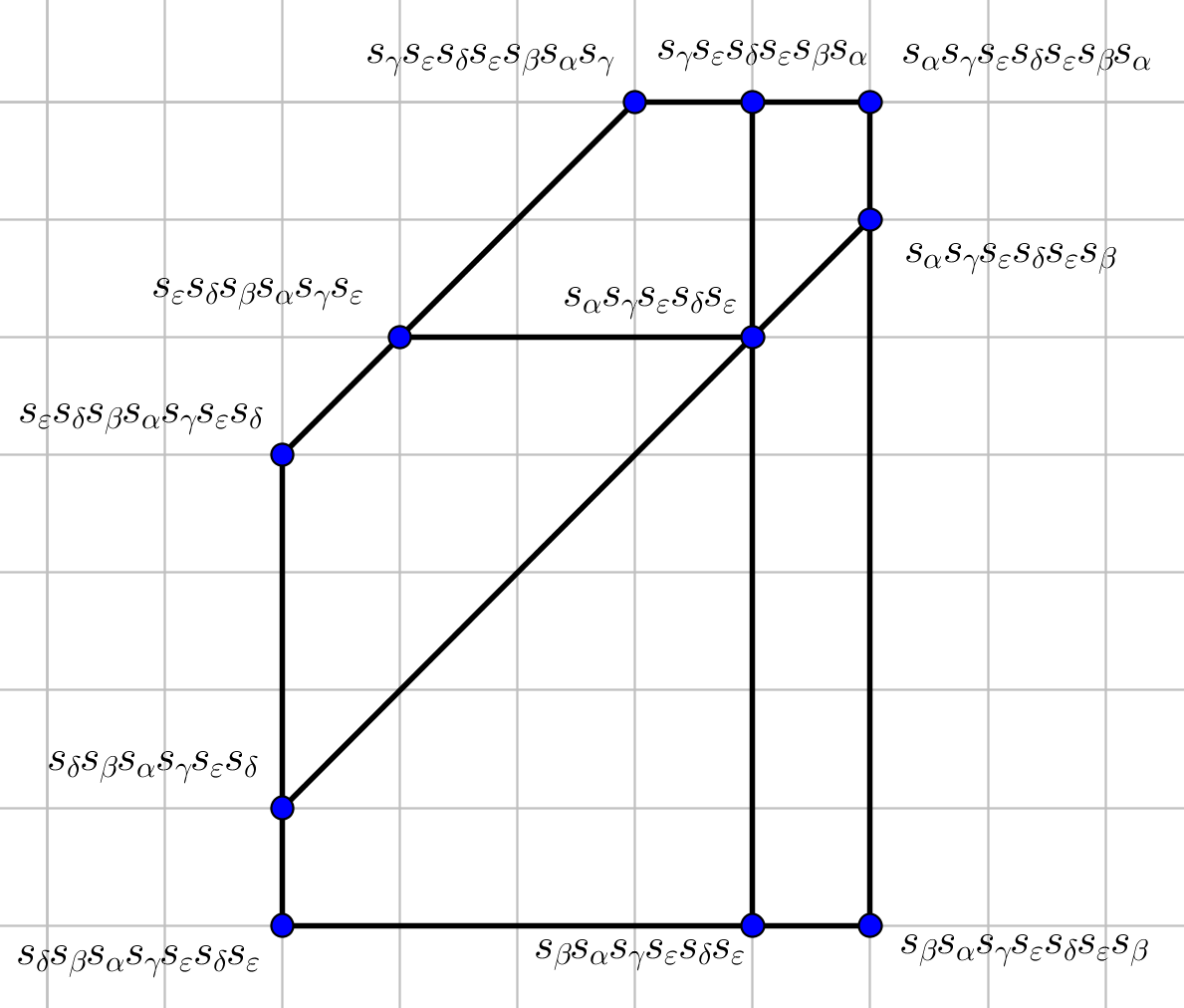}
\end{figure}
\end{minipage}

\begin{minipage}{\textwidth}\item The pizza $[A_2^{opp}, A_2^{opp}, A_2^{opp}, A_2^{opp}b, A_2]$
\begin{figure}[H]
\centering
\begin{subfigure}{0.7\textwidth}
\includegraphics{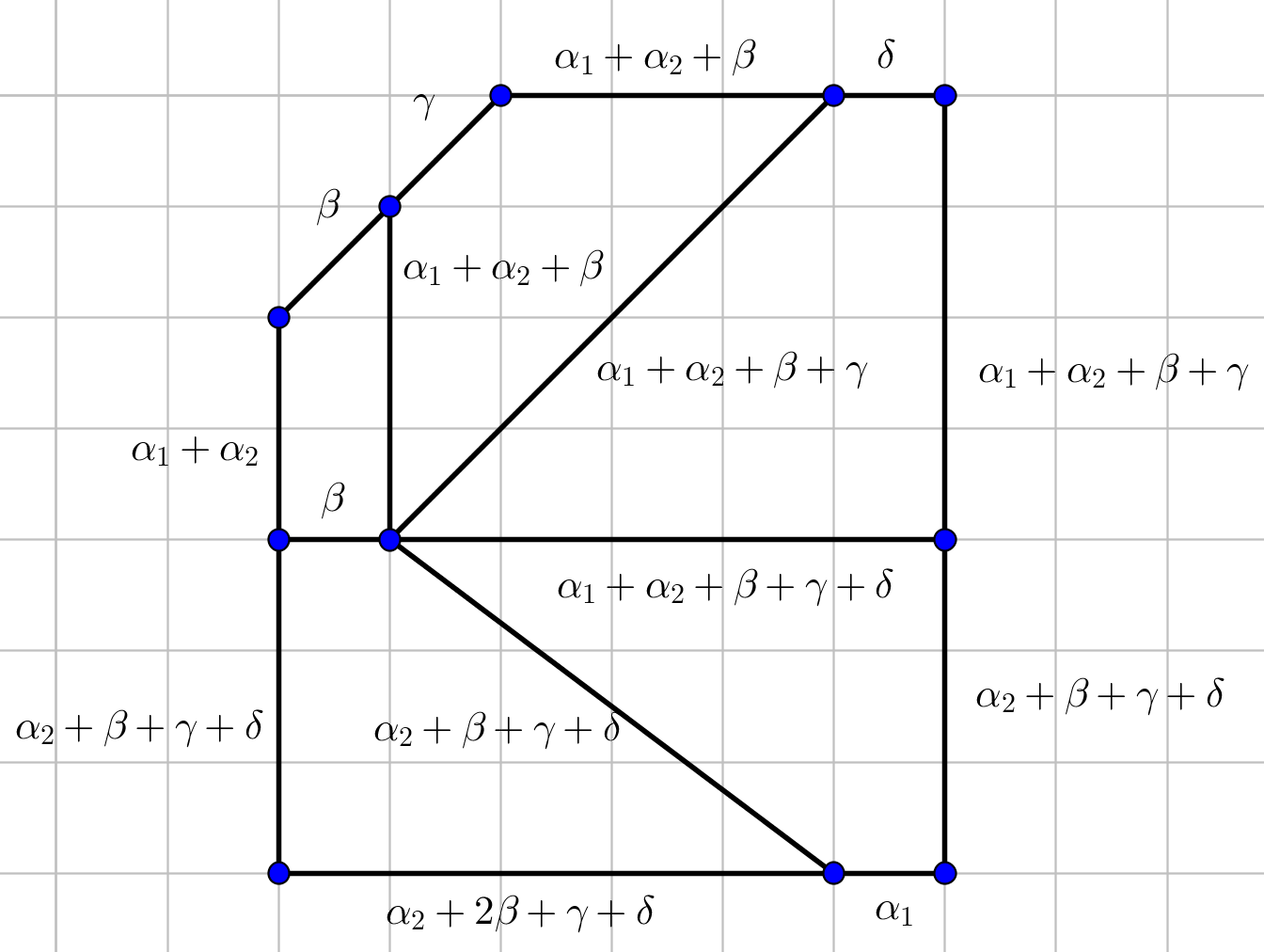}
\end{subfigure}
\begin{subfigure}{0.29\textwidth}
\includegraphics{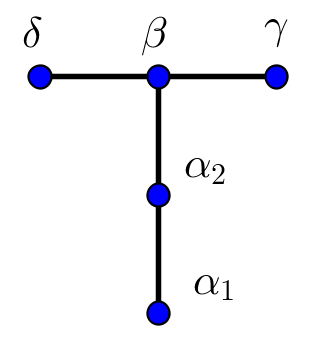}
\end{subfigure}
\end{figure}
and the $W$-elements:
\begin{figure}[H]
\centering
\includegraphics{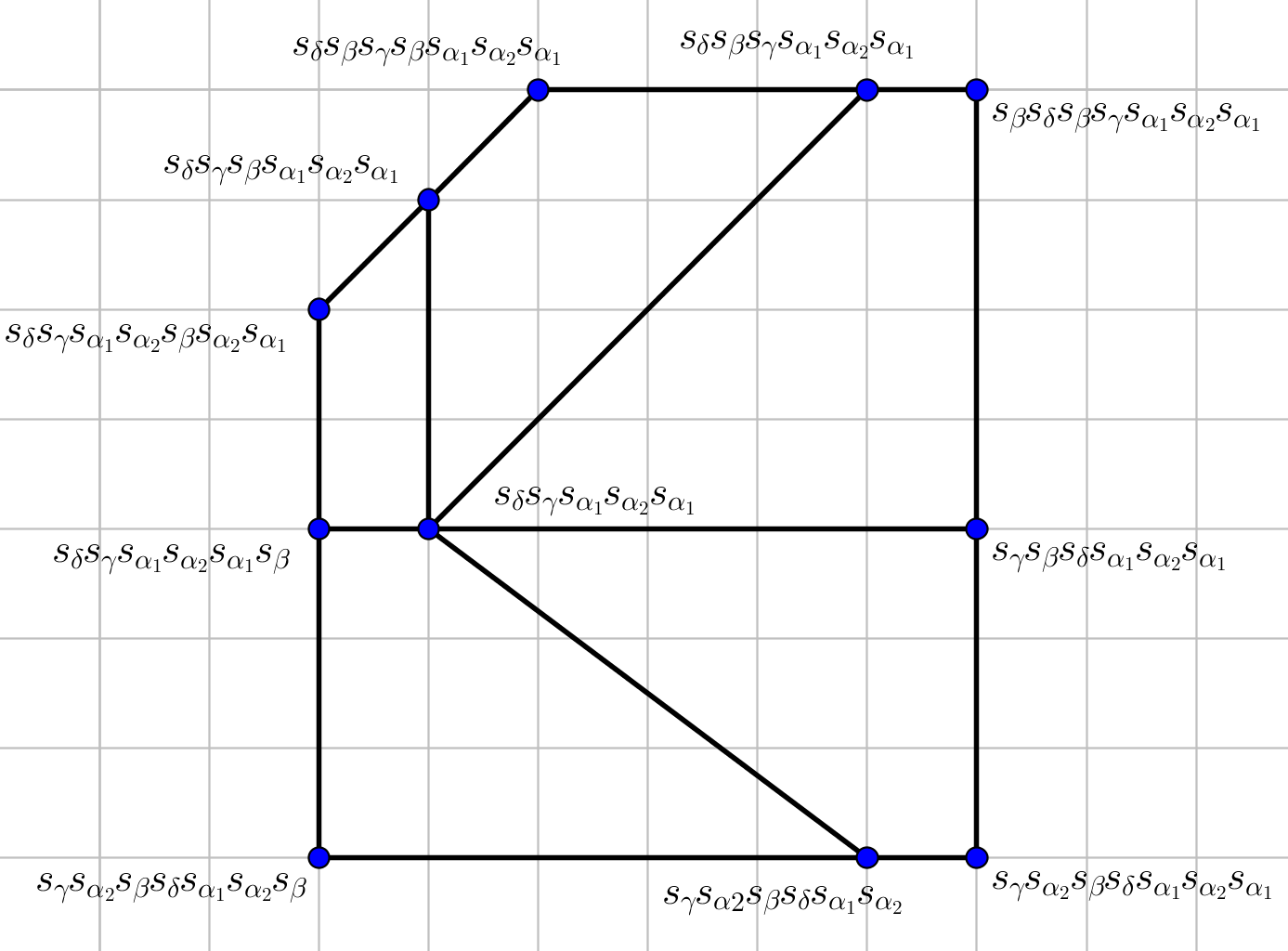}
\end{figure}
\end{minipage}

\subsection*{Height $9$}

\begin{minipage}{\textwidth}\item The pizza $[A_2, A_2^{\text{opp}}b, A_2^{\text{opp}}, A_2^{\text{opp}}, A_2^{\text{opp}}]$:
\begin{figure}[H]
\centering
\begin{subfigure}{0.6\textwidth}
\includegraphics{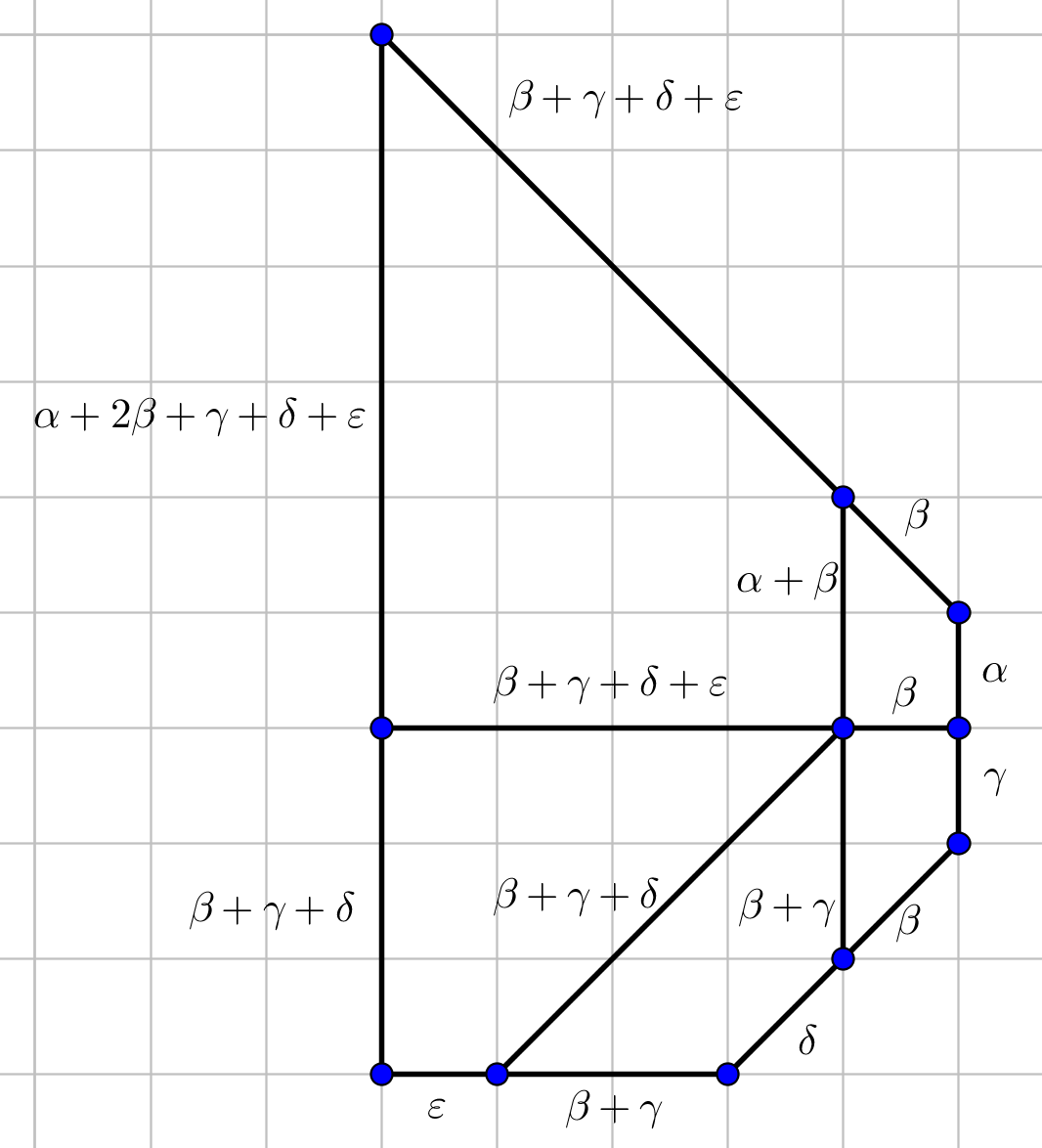}
\end{subfigure}
\begin{subfigure}{0.39\textwidth}
\includegraphics{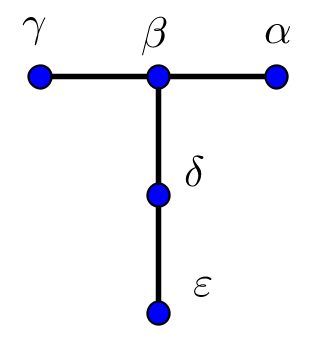}
\end{subfigure}
\end{figure}
\begin{figure}[H]
\centering
\includegraphics{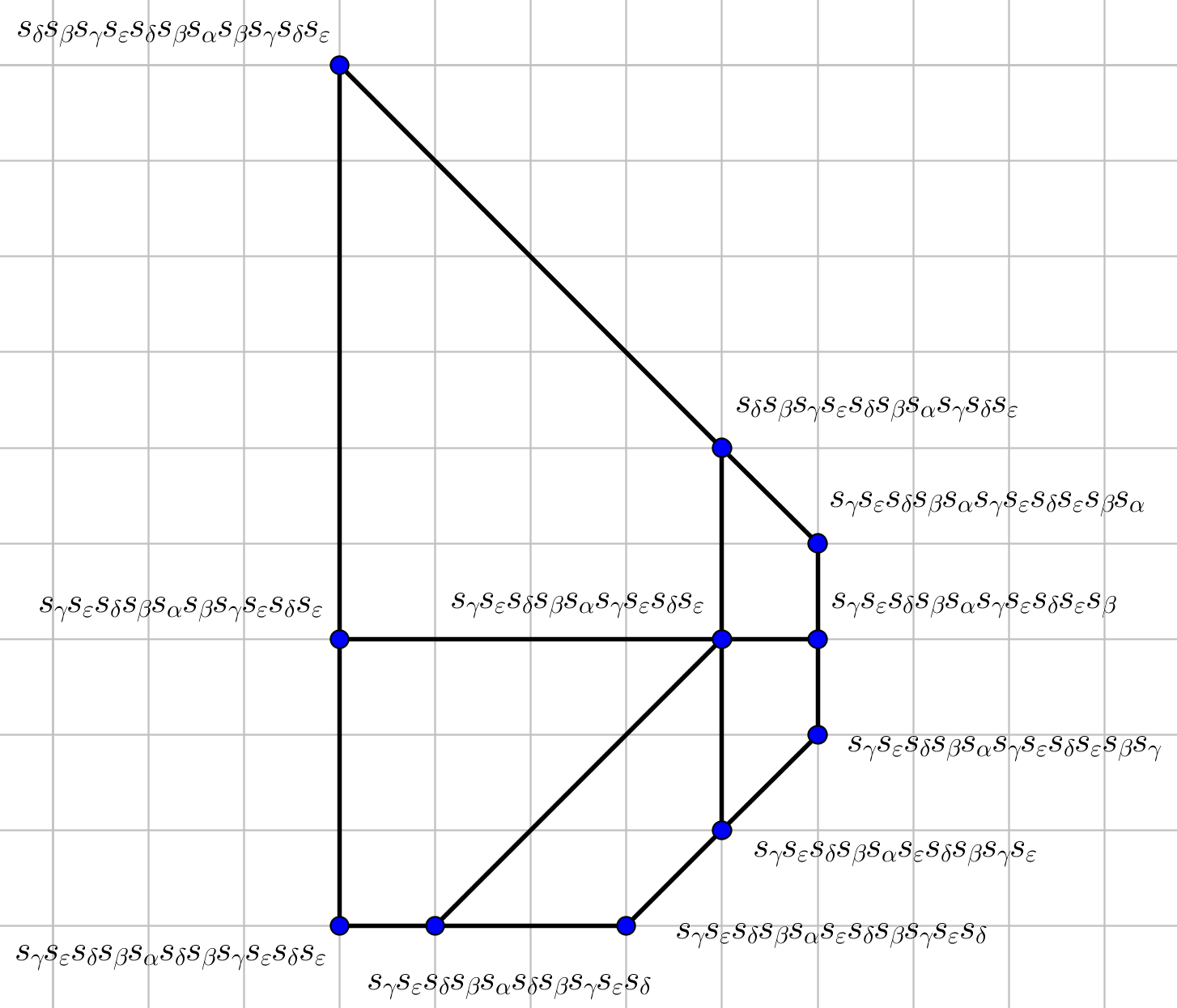}
\end{figure}
\end{minipage}

\end{enumerate}
\section{Embedded degenerations}\label{sec:embeddeddegens}
Here we will describe an embedded degeneration (from definition \ref{def:KLatlas}) of a smooth toric surface $M$ into the union of Richardson varieties. We will find a point $x$ in $H/B_H$ such that $\overline{T_M\cdot x}\cong M$. We will do this by determining which Pl\"{u}cker coordinates vanish. Consider the following diagram:
\[
\xymatrix{
M\ar@{^{(}->}[d]\ar[rr]^{\Phi_{T_H}} & & \Phi_{T_H}(M)\ar@{^{(}->}[d]\ar[r]^{\Phi_{T_M}} & \Phi_{T_M}(M)\ar@{^{(}->}[d] \\
H/B_H\ar@{->>}[r] & H/P^{\alpha}\ar^{\Phi_{T_H}}[r] & Q\ar^{\Phi_{T_M}}[r] & \Phi_{T_M}(Q)
}
\]
where $P^{\alpha}$ is the maximal proper parabolic not containing the subgroup corresponding to $-\alpha$, $Q=\Phi_{T_H}(H/P^{\alpha})$ is the moment polytope of $H/P^{\alpha}$, and $\Phi_{T_M}:\ft^*_H\to\ft^*_M$ is the map induced by $T_M\subseteq T_H$. Each of the vertices $\lambda$ of $\Phi_{T_H}(Q)$ corresponds to a Pl\"{u}cker coordinate, and if $\Phi_{T_M}(\lambda)\notin \Phi_{T_M}(M)$, then we know that the $\lambda$ Pl\"{u}cker coordinate should vanish on $M$. To find an embedding $M\hookrightarrow H/B_H$, we just need to find an element $x\in H/B_H$ for which exactly these Pl\"{u}cker coordinates vanish, and take $\overline{T_M\cdot x}\subseteq H/B_H$.

We go through an example. Consider the pizza
\begin{figure}[H]
\centering
\includegraphics{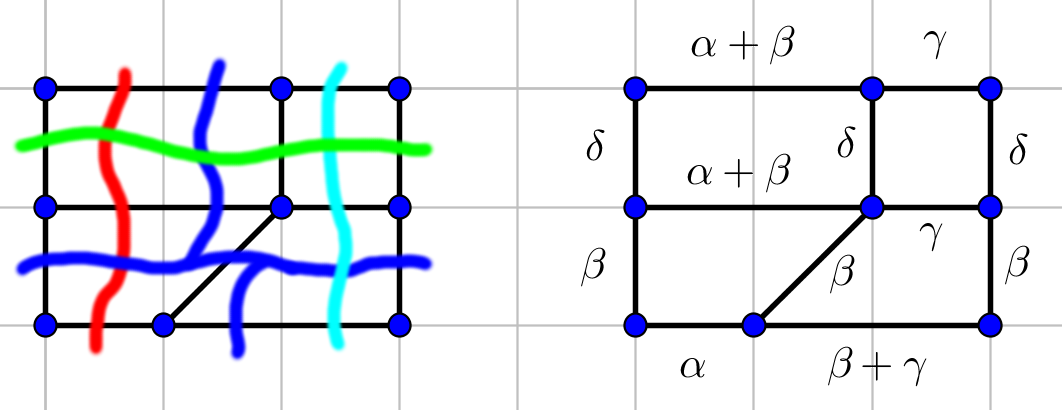}
\end{figure}
\noindent and $H$ with diagram
\begin{figure}[H]
\centering
\includegraphics{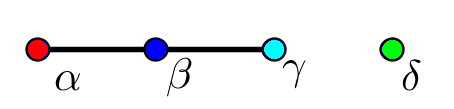}
\end{figure}
\noindent and labels (written in one line notation by the identification $W=S_4\times S_2$)
\begin{figure}[H]
\centering
\includegraphics{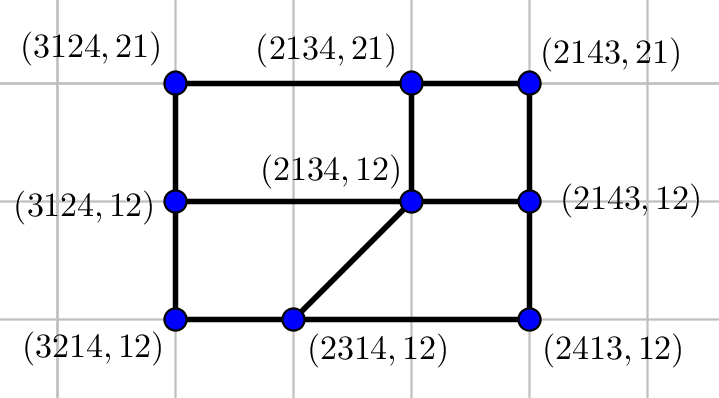}
\end{figure}
\noindent We relabel the edges in a way that they correspond to the left-multiplications (action on values as opposed to positions) in $W$:
\begin{figure}[H]
\centering
\includegraphics{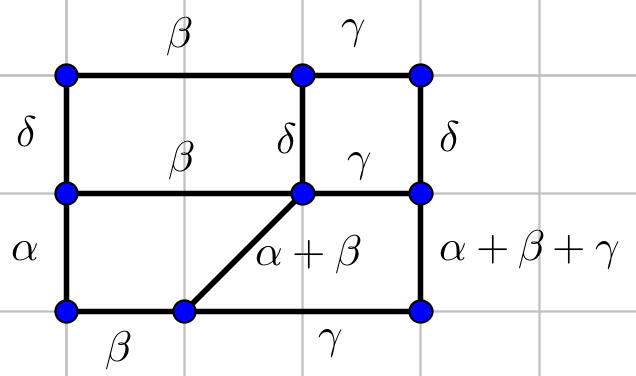}
\end{figure}
\noindent Then read off the directions that left-multiplication by simple roots correspond to
\begin{figure}[H]
\centering
\includegraphics{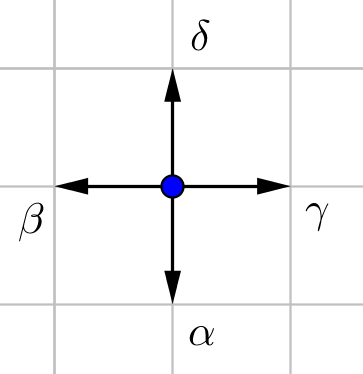}
\caption{$T_M\subset T_H$\label{fig:TMinsideT}}
\end{figure}
\noindent Note that this expresses how the subtorus $T_M$ sits in $T_H$. Then for all four simple roots, we contract the edges of the pizza except the ones whose label contains the chosen simple root as a summand, and see which Pl\"{u}cker coordinates lie outside the polytope:
\begin{figure}[H]
\centering
\includegraphics{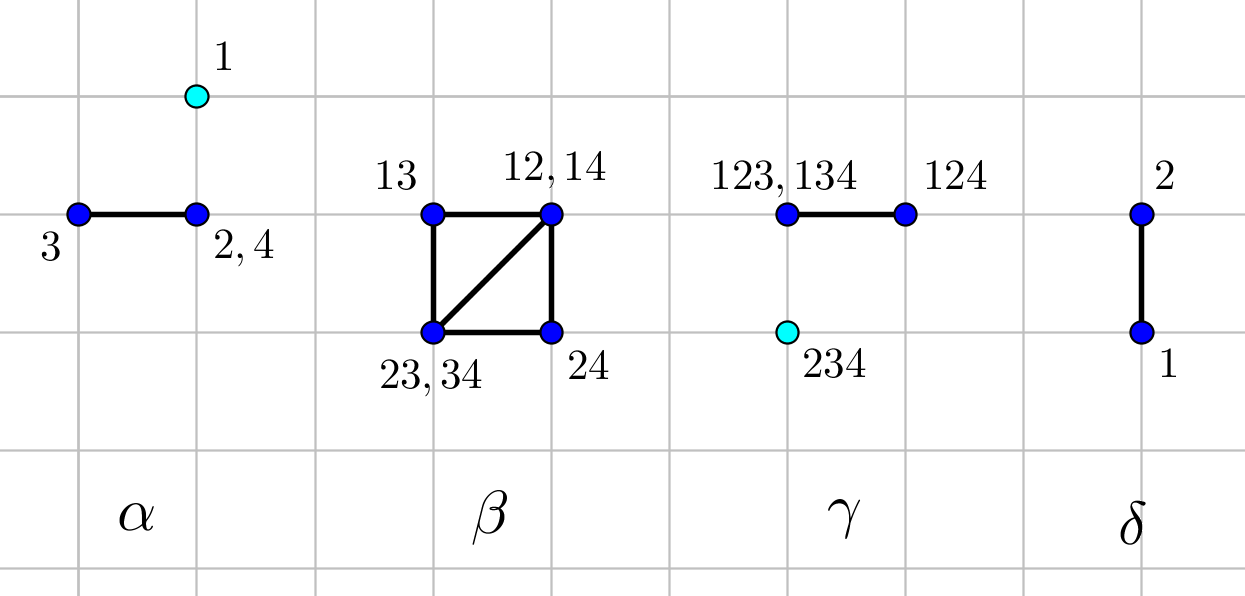}
\end{figure}
So the vanishing Pl\"{u}cker coordinates are: $(1,-)$ and $(234,-)$. A representative in $H$ for which precisely these Pl\"{u}cker coordinates vanish is the pair of matrices
\[
r=
\left(
\begin{pmatrix}
0 & 1 & 0 & 0 \\
1 & 0 & 1 & 0 \\
1 & 1 & 2 & 0 \\
1 & 2 & 3 & -1
\end{pmatrix},
\begin{pmatrix}
1 & 0 \\
1 & 1 
\end{pmatrix}
\right),
\]
where $(1,-)$ is the $(1,1)$-entry of the first matrix, and $(234,-)$ is the minor formed by the columns $1,2,3$ and the rows $2,3,4$ of the first matrix.
A parametrization for $T_M$ is (by figure \ref{fig:TMinsideT})
\[
\left(
\begin{pmatrix}
a & 0 & 0 & 0 \\
0 & ab^{-2} & 0 & 0 \\
0 & 0 & a^{-3}b^4 & 0 \\
0 & 0 & 0 & ab^{-2}
\end{pmatrix},
\begin{pmatrix}
b^{-1} & 0 \\
0 & b
\end{pmatrix}
\right),
\]
with $a,b\in \bC^\times$, so we have $M \cong \overline{T_M\cdot rB_H/B_H}$.



\end{document}